\documentclass[11pt,a4paper]{article}

\usepackage[left=2.45cm, top=2.45cm,bottom=2.45cm,right=2.45cm]{geometry}

\usepackage{amsmath,amssymb,amsthm,mathrsfs,calc,graphicx,stmaryrd,xcolor}
\usepackage[british]{babel}
\usepackage{amsfonts}              % for blackboard bold, etc

\newtheoremstyle{definition}
{10pt}% measure of space to leave above the theorem. E.g.: 3pt
{10pt}% measure of space to leave below the theorem. E.g.: 3pt
{}% name of font to use in the body of the theorem
{}% measure of space to indent
{\bfseries}% name of head font
{}% punctuation between head and body
{.5em}% space after theorem head; " " = normal interword space
{}% Manually specify head

\newtheoremstyle{plain}
{10pt}% measure of space to leave above the theorem. E.g.: 3pt
{10pt}% measure of space to leave below the theorem. E.g.: 3pt
{\itshape}% name of font to use in the body of the theorem
{}% measure of space to indent
{\bfseries}% name of head font
{}% punctuation between head and body
{.5em}% space after theorem head; " " = normal interword space
{}% Manually specify head

\theoremstyle{plain}	
\newtheorem{thm}{Theorem}[section]
\newtheorem{lem}[thm]{Lemma}

\newtheorem{cor}[thm]{Corollary}

\theoremstyle{definition}	

\newtheorem{remark}[thm]{Remark}

\newcommand{\RR}{\mathbb{R}}      % for Real numbers
\newcommand{\origin}{\boldsymbol{o}}

\def\BB{\mathbb{B}}

\def\EE{\mathbb{E}}

\def\NN{\mathbb{N}}
\def\PP{\mathbb{P}}

\def\RR{\mathbb{R}}

\def\BBd{\mathbb{B}^d}
\def\SSd{\mathbb{S}^{d-1}}

\def\a{\alpha}
\def\be{\beta}

\def\la{\lambda}

\def\cC{\mathcal{C}}

\def\cH{\mathcal{H}}

\def\cP{\mathcal{P}}

\def\dint{\textup{d}}

\def\var{{\textup{var}}}

\def\ext{{\rm ext}}

\allowdisplaybreaks

\begin{document}

\title{\bfseries Universal Scaling Limits for\\ Generalized Gamma Polytopes}

\author{Julian Grote\footnotemark[1]}

\date{}
\renewcommand{\thefootnote}{\fnsymbol{footnote}}
\footnotetext[1]{Ruhr University Bochum, Faculty of Mathematics, D-44780 Bochum, Germany. E-mail: julian.grote@rub.de. The author has been supported by the Deutsche Forschungsgemeinschaft (DFG) via RTG 2131 \emph{High-Dimensional Phenomena in Probability -- Fluctuations and Discontinuity.}}

\maketitle

\begin{abstract}
Fix a space dimension $d\ge 2$, parameters $\a > -1$ and $\beta \ge 1$, and let 
%$K_\la$ be the random convex hull of a Poisson point process in $\RR^d$
%$n\in \NN$ random points in $\RR^d$, independently and identically distributed according to an isotropic random variable in $\RR^d$ with density proportional to 
$\gamma_{d,\a, \beta}$ be the probability measure of an isotropic random vector in $\RR^d$ with density proportional to
\begin{align*}
||x||^\a\,  \exp\left(-\frac{\|x\|^\beta}{\beta}\right), \qquad x\in \RR^d.
\end{align*}
By $K_\lambda$, we denote the Generalized Gamma Polytope arising as the random convex hull of a Poisson point process in $\RR^d$ with intensity measure $\lambda\gamma_{d,\a,\beta}$, $\lambda>0$. 
%In the Gaussian case, i.e., $\alpha = 0$ and $\beta = 2$, Calka and Yukich \cite{CalkaYukich} established that the scaling limit of the boundary of $K_n$ 
We establish that the scaling limit of the boundary of $K_\la$, as $\la \rightarrow \infty$, is given by a universal `festoon' of piecewise parabolic surfaces, %having apices at a Poisson point process in the product space $\RR^{d-1} \times \RR$ that has density 
%\begin{align}\label{bitch}
%(v,h) \mapsto e^h\,,\qquad  (v,h) \in \RR^{d-1} \times \RR\,,
%\end{align}
%with respect to the Lebesgue measure on $\RR^{d-1}\times \RR$, 
independent of $\alpha$ and $\beta$. 
%Thus, this festoon turns out to be a unique scaling limit for the boundary of our huge class of isotropic random polytopes. 
%Moreover, we show expectation and variance asymptotics for the intrinsic volumes and face numbers of $K_n$.
%that can also be expressed in terms of the Poisson point process introduced in \eqref{bitch}.

Moreover, we state a list of other large scale asymptotic results, including expectation and variance asymptotics, central limit theorems, concentration inequalities, Marcinkiewicz-Zygmund-type strong laws of large numbers, as well as moderate deviation principles for the intrinsic volumes and face numbers of $K_\la$. 
\end{abstract}

%\tableofcontents

\section{Introduction and main result}

We analyze the class of Generalized Gamma Polytopes, defined as the random convex hulls of a Poisson point process, whose intensity measure is given by a multiple of a huge class of isotropic measures on $\RR^d$, $d\ge 2$, including the Gaussian one as a special case. 
Specifically, such a random polytope is constructed in three steps.\\ 
First, let $N$ be a Poisson distributed random variable of intensity $\lambda > 0$, i.e., 
%\begin{align*}
$\PP(N=k) = (\lambda^k e^{-\lambda})/k!,\, k \in \NN_0.$
%\end{align*}
Secondly, fix $\alpha > -1$ and $\beta \ge 1$, and choose a random number of $N$ points in $\RR^d$, independently and distributed according to the density
\begin{align*}%\label{DefinitionPhi}
\phi_{\a , \beta}(x):= c_{\a,\beta}^d\, ||x||^\a\,  \exp\left(-\frac{\|x\|^\beta}{\beta}\right) :=   \left(\frac{\beta^{\frac{\beta - \a -1}{\beta}}}{2\, \Gamma\left(\frac{\a + 1}{\beta}\right)}\right)^d \, ||x||^\a \, \exp\left(-\frac{\|x\|^\beta}{\beta}\right), \qquad x \in \RR^d.
\end{align*}
We denote this point set by $\mathcal{P}_\lambda$. In a third step, the random convex hull of $\mathcal{P}_\lambda$, indicated by $K_\lambda$, defines the Generalized Gamma Polytope.\\
The family of stated densities can be summarized under the class of the generalized Gamma distribution, giving rise to the description Generalized Gamma Polytopes. As special cases, it includes the Gaussian distribution $(\a=0, \beta =2)$, the generalized normal distribution $(\a=0, \beta\ge 1)$, the Gamma distribution $(\a \ge 0, \beta = 1)$ and the Weibull distribution $(\a > 0, \beta = \a + 1)$.\\

\begin{figure}[t] 
	\centering
	\includegraphics[width=\textwidth]{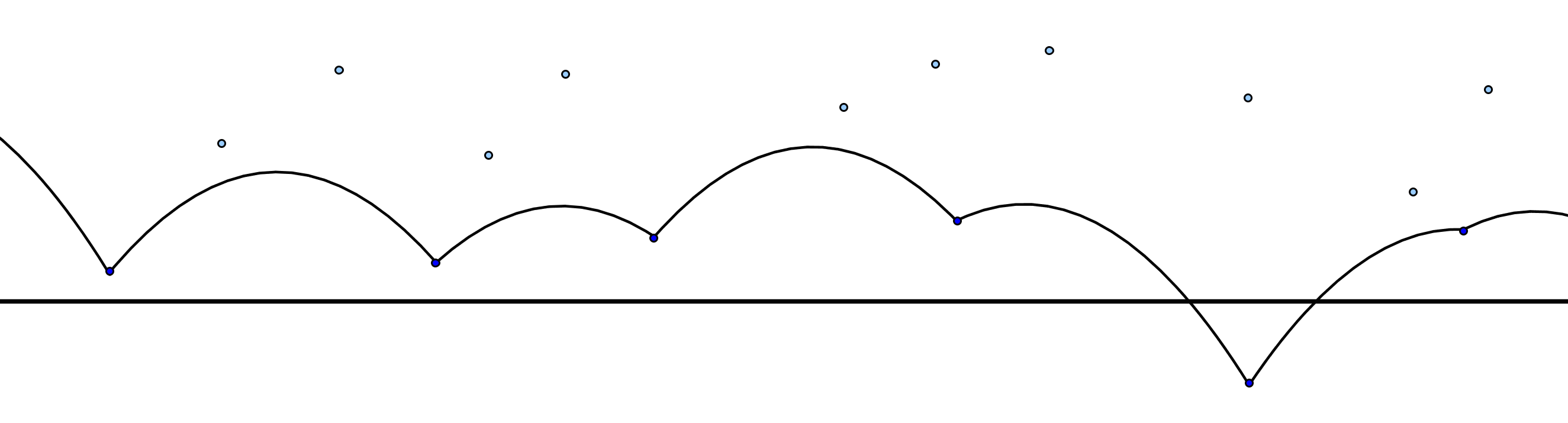}
	\caption{The boundary of $\Phi$. The process $\cP$; its extreme points are marked in dark blue.}
	\label{Limit.}
\end{figure}
%\subsection{The Gaussian case}
In the Gaussian setup, i.e., if $\alpha = 0$ and $\beta = 2$, the induced Gaussian random polytope is a well-studied object in literature. One reason for the interest lies in its applications to other fields of mathematics. For example, Gaussian polytopes are highly relevant in asymptotic convex geometry or the local theory of Banach spaces (see Gluskin \cite{Gluskin}), they are prototypical examples of random convex sets that satisfy the (probabilistic version of the) celebrated hyperplane conjecture (see Klartag and Kozma \cite{KK}), and show a clear relevance also in the area of multivariate statistics (see Cascos \cite{Cascos}). For more details, we refer to the surveys about random polytopes by B\'ar\'any \cite{BaranySurvey}, Hug \cite{HugSurvey} and Reitzner \cite{ReitznerSurvey}.\\
Regarding to this Gaussian case, Calka and Yukich \cite[Theorem 1.2]{CalkaYukich} established that the scaling limit of the boundary of $K_\la$ converges to a `festoon' of piecewise parabolic surfaces, having apices at the points of a Poisson point process $\cP$ on the product space $\RR^{d-1} \times \RR$, whose intensity measure has density  
\begin{align}\label{qqqq}
(v,h) \mapsto e^h,\qquad  (v,h) \in \RR^{d-1} \times \RR,
\end{align}
with respect to the Lebesgue measure on $\RR^{d-1} \times \RR$ (see Figure \ref{Limit.}).
%at a Poisson point process on $\RR^{d-1} \times \RR$ with intensity $e^h\dint h \dint v$.\\
In the main theorem of the present paper, we generalize the result from \cite{CalkaYukich} to the situation where the underlying random polytope is given by our much broader class of Generalized Gamma Polytopes. In order to formulate the main result corresponding to the universality of the scaling limit of the boundary of $K_\la$, we first need to introduce some preparations, and start by modifying the crucial scaling transformation, introduced in \cite[Equation (1.5)]{CalkaYukich} in the Gaussian setting, to our purpose.\\ 
We work in the Euclidean space $\RR^d$ of dimension $d\ge 2$ with origin $\origin$ and north pole $u_0:= (0,\ldots,0,1)$ on the $(d-1)$-dimensional unit sphere $\SSd$. For $x,y\in \RR^d$, we denote by $\langle \cdot , \cdot \rangle$ the standard scalar product with associated Euclidean norm $||\cdot||$. Moreover, let $\BB^d(x,r)$ be the closed ball centered at $x\in \RR^d$ with radius $r>0$.
If $T_{u_0}:= T_{u_0}(\mathbb{S}^{d-1})$ is the tangent space at the north pole, we identify $T_{u_0}$ with the $(d-1)$-dimensional Euclidean space $\RR^{d-1}$. Besides, we define $\exp^{-1}$ as the inverse of the exponential map $\exp := \exp_{u_0}: T_{u_0} \rightarrow \mathbb{S}^{d-1}$. It maps a vector $v\in T_{u_0}$ to the point $u\in \mathbb{S}^{d-1}$ in such a way that $u$ lies at the end of the unique geodesic ray with length $\|v\|$, emanating at $u_0$ and having direction $v$. Note that the exponential map is injective on $\BB_{d-1}(\origin,\pi) := \{v\in T_{u_0}: \|v\| < \pi\}$ and we have that $\exp(\BB_{d-1}(\origin,\pi)) = \mathbb{S}^{d-1}\setminus\{-u_0\}$. (Following \cite{CalkaYukich,GroteThäle}, we prefer \index{$\BB_{d-1}(\origin,r)$} to write $\BB_{d-1}(\origin,r)$ for a centered ball of radius $r>0$ in $T_{u_0}$ instead of $\BB^{d-1}(\origin,r)$ to prevent confusions.) Since the inverse of the exponential map is well-defined on the whole sphere $\mathbb{S}^{d-1}$, except for the point $-u_0$, we put $\exp^{-1}(-u_0):=(\origin,\pi)$.\\
Now, for sufficiently large $\la$, the Generalized Gamma Polytope $K_\la$ can be expected to grow like $\BBd(\origin,(\be \log \la)^{1/\be})$ (see Remark \ref{criticalradius} (a)). In order to reflect this behavior in our scaling \index{$R_\lambda$} transformation, define, for all $\la > 0$ such that $R_\la \ge 1$,
\begin{align*}
R_\la := \left[\be \log \la- \left(\frac{\be(d+1) - 2d - 2\a}{2}\right) \log \left(c_{\a,\be}^{- \frac{2\be d}{\be(d+1) - 2d - 2\a}} \be \log \la \right)\right]^{\frac{1}{\be}}.
\end{align*}
In particular, $R_\lambda$ is asymptotically equivalent to the critical radius $(\be \log \la)^{1/\be}$ itself.  
The reason for this explicit choice of $R_\la$ will become clear in the proof of the upcoming Lemma \ref{LemmaIntensityP^lambda}. We are now in the position to define the scaling transformation (see Figure \ref{Transformation.} for an illustration).\\
%\begin{definition}
The \index{scaling transformation} mapping $T_\la : \RR^d \rightarrow \RR^{d-1} \times \RR$, defined \index{$T_\lambda$} by  
\begin{figure}[t] 
	\centering
	\includegraphics[width=\textwidth]{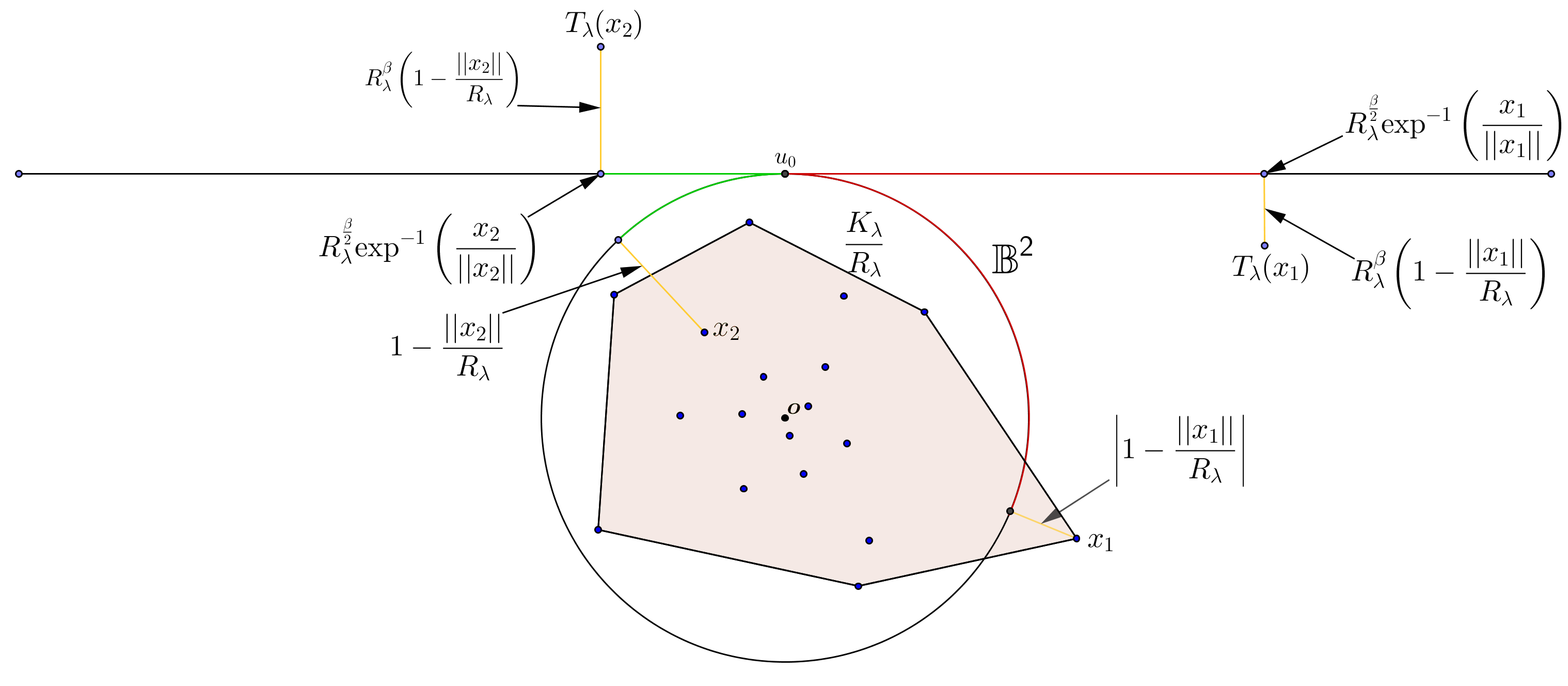}
	\caption{The scaling transformation $T_\lambda$.}
	\label{Transformation.}
\end{figure}
\begin{align*}
T_\la(x) := \left(R_\la^{\frac{\beta}{2}}\, \exp^{-1}\left(\frac{x}{\|x\|}\right), R_\la^\be\, \left(1-\frac{\|x\|}{R_\la}\right)\right),\qquad x\in\RR^d\setminus\{\origin\},
\end{align*}
maps $\RR^d\setminus\{\origin\}$ into the \index{$W_\lambda$} region 
\begin{align*}
W_\la :=R_\la^{\frac{\beta}{2}}\, \BB_{d-1}(\origin,\pi) \times (-\infty, R_\la^\be]\subseteq \RR^{d-1} \times \RR.
\end{align*}
Putting $T_\la(\origin) := (\origin,R_\la^\be)$, the transformation $T_\la$ is a bijection between $\RR^d$ and $W_\la$.\\
Moreover, letting
\begin{align*}
\Pi^{\downarrow}:=\left\{(v,h)\in\RR^{d-1}\times\RR:h\leq- \frac{\|v\|^2}{2}\right\}
\end{align*}
be the unit downward paraboloid, we define the limiting germ-grain process $\Phi$ by
\begin{align}\label{Limitprocess}
\Phi := \Phi(\cP) := \bigcup_{w\in \RR^{d-1} \times \RR \atop \cP\, \cap\, \text{int} (\Pi^\downarrow(w))=\emptyset} [\Pi^\downarrow(w)]^{(\infty)},
\end{align}
where we recall the definition of the Poisson point process $\cP$ from \eqref{qqqq}, and $\text{int}(\cdot)$ denotes the interior of the argument set. Here, for $w:=(v,h)\in \RR^{d-1}\times \RR$, we put
\begin{align*}
[\Pi^\downarrow(w)]^{(\infty)} := w \oplus \Pi^{\downarrow},
\end{align*}
where $\oplus$ is the usual Minkowski sum.
All the points of $\cP$ that belong \index{$[\Pi^\uparrow(w)]^{(\infty)}$} to the \index{$[\Pi^\downarrow(w)]^{(\infty)}$} boundary of $\Phi$ are summarized in the set of extreme points of $\cP$, denoted \index{$\ext(\cP)$} by $\ext (\cP)$
(see Figure \ref{Limit.}).\\

We are finally able to state our main theorem. In particular, it formalizes that the re-scaled configuration of vertices of $K_\la$ converges to the set $\ext (\cP)$, and that the scaling limit of the boundary of $K_\la$ arises as the boundary of the germ-grain process $\Phi$, as $\la \rightarrow \infty$, independent of the parameter $\alpha$ and $\beta$ in the underlying density function.\\ 
Let $\cC (\BB^d(x,r))$ be the space of all continuous functions on $\BB^d(x,r)$, equipped with the supremum norm. 
\newpage
\begin{thm}\label{Festoon}
	Fix $L\in (0,\infty)$. As $\la \rightarrow \infty$, the following assertions are true.
	\begin{enumerate}
		\item[(a)] Under the scaling transformation $T_\la$, the rescaled set of vertices of $K_\la$ converges in distribution to the set of extreme points of $\cP$.
		\item[(b)] Under the scaling transformation $T_\la$, the rescaled boundary of $K_\la$ converges in probability to the boundary of $\Phi$, on the space $\cC(\BB_{d-1}(\origin,L))$.
	\end{enumerate} 
\end{thm}

The rest of this paper is structured as follows. In Section 2, we present the proof of Theorem \ref{Festoon}. The final Section 3 has a slightly different focus and is concerned with a huge variety of large scale asymptotic results for the intrinsic volumes and face numbers of the Generalized Gamma Polytope $K_\la$.

\section{Proof of the main result}

Let us briefly recall the general setup. By $\cP_\la$, we denote a Poisson point process in $\RR^d$, whose intensity measure is a multiple $\la>0$ of the measure $\gamma_{d,\a,\be}$. The Generalized Gamma Polytope $K_\la$ is defined as the random convex hull generated by $\cP_{\la}$. %For the sake of readability, we have decided to split the proof of the main theorem into several steps. 
We start the proof of Theorem \ref{Festoon} by analyzing the scaling transformation $T_\la$ and its properties.

\subsection{Scaling transformation}\label{sec:scaling}

In the Gaussian case, i.e., $\a = 0$ and $\be=2$, it follows from the work of Geffroy \cite{Geffroy} that the Hausdorff distance between $K_{\la_k}$ and $\BB^d(\origin,\sqrt{2 \log \la_k})$ converges to 0 almost surely, as $k \to\infty$, along `suitable' subsequences $\la_k$ tending to infinity. The first goal of this section is to determine this critical ball in our generalized setting, following from the next lemma (see also \cite[Page 155]{Mikosch} for a slightly different statement). It can be proved by using standard tools from extreme value theory. 
%(for a detailed proof, see \cite[Proof of Theorem 3.1.1]{DissGrote}).
\begin{lem}
	Let $\a > -1$, $\be \ge 1$, and let $X_1,X_2,\ldots$ be independent random variables in $\RR$, distributed according to the density
	\begin{align*}
	f_{\a,\be}(x):= c_{\a,\be}\, |x|^\a\, \exp\left(-\frac{|x|^\be}{\be}\right) := \frac{\be^{\frac{\be - \a -1}{\be}}}{2\, \Gamma\left(\frac{\a + 1}{\be}\right)}\, |x|^\a\, \exp\left(-\frac{|x|^\be}{\be}\right), \qquad x \in \RR.
	\end{align*}
	Put $M_n:=\max\{X_1,\ldots,X_n\}$, $n\in \NN$. Then, for \index{$f_{\alpha,\beta}(x)$} all $x\in \RR$, it holds that 
	\begin{align*}
	&\lim\limits_{n\rightarrow \infty} \PP\left( (\be\log n)^{\frac{\be-1}{\be}} \left[M_n - \left((\be\log n)^{\frac{1}{\be}} - \frac{(\beta - \alpha - 1)\, \log \left(c_{\alpha,\be}^{-\frac{\be}{\be - \a -1}}\, \be\log n\right)}{\beta\, (\be\log n)^{\frac{\be-1}{\be}}}\right)\right] \le x\right)\\
	&\qquad \qquad = \exp(-e^{-x}).
	\end{align*} 
\end{lem}

\begin{remark}\label{criticalradius}
\begin{itemize}
\item[(a)] Loosely speaking, the previous result yields that for all $\a > -1$, $\be\ge 1$ and sufficiently large $n$, the maximum $M_n$ takes values that are `close' to $(\be \log n)^{1/\be}$, independent of the second parameter $\a$. Moreover, the difference between $M_n$ and $(\be \log n)^{1/\be}$ is random and of the magnitude $(\be\log n)^{- (\be-1)/\be}$.
In our Poissonized model, this indicates that $(\be \log \la)^{1/\be}$ should be chosen as the critical radius, i.e., $K_\la$ can be expected to grow like $\BB^d(\origin, (\be \log \la)^{1/\be})$, for all $\be\ge 1$ and $\a > -1$, as $\la \rightarrow \infty$. 
\item[(b)] Since $(\be\log n)^{- (\be-1)/\be}$ tends to infinity, if $\beta < 1$, we may and will restrict to the condition $\beta \ge 1$. This natural condition was used also by Carnal \cite[Page 171]{Carnal} and Eddy and Gale \cite[Page 757]{Eddy2}.  
%\item[(c)] By using the method described in \cite{Geffroy}, it seems likely to show that also the Hausdorff distance between $K_{\lambda_k}$ and $\BB^d(\origin, (\be \log \la_k)^{1/\be})$ converges to $0$ almost surely, as $k \rightarrow \infty$, along `suitable' subsequences $\la_k$. We leave this issue to further research. 
\end{itemize}
\end{remark}

%\begin{remark}
%	Later, it turns out to be crucial to bound the exponential term in \eqref{IntensityP^lambda} uniformly by $e^h$, \textit{for all} $h\in \RR$. Examples are provided by the estimates presuming \eqref{Bound1}, \eqref{Bound2}, \eqref{UngleichungDichte}, \eqref{eq:DefPhiLambda2Variablen} and \eqref{phiuntenlambda}. However, if $\beta < 1$, this is not achievable and, therefore, we may and will restrict to the condition $\beta \ge 1$. This natural condition was used also by Carnal \cite[Page 171]{Carnal} and Eddy and Gale \cite[Page 757]{Eddy2}.  
%\end{remark}

Now, define the rescaled point process \index{$\mathcal{P}^{(\lambda)}$} by 
%\begin{align*}
$\cP^{(\la)} := T_\la(\cP_\la)$
%\end{align*}
(see Figure \ref{Transformation..}). Due to the mapping property for Poisson point processes (see, for example, \cite[Theorem 5.1]{Last}), the point process $\cP^{(\la)}$ is actually also a Poisson point process in its target region $W_\la$. Its distributional properties will be analyzed in the following two statements. 

\begin{lem}\label{LemmaIntensityP^lambda}
	The intensity measure of $\cP^{(\la)}$ has density
	\begin{align}\label{IntensityP^lambda}
	\begin{split}
	(v,h) &\mapsto \frac{\sin^{d-2} (R_\la^{-\frac{\beta}{2}}\|v\|)}{\|R_\la^{-\frac{\beta}{2}}v\|^{d-2}}\, \frac{(\be\log \la)^{\frac{\be(d+1) - 2d - 2\a}{2\be}}}{R_\la^{\frac{\be(d+1) - 2d - 2\a}{2}}}\\
	&\qquad \times \exp\left(h - \frac{h^2}{2 R_\la^\be} (\be-1) (1-C)^{\be-2}\right)
	\left(1-\frac{h}{R_\la^\be}\right)^{d-1+\a}\,{\bf 1}((v,h)\in W_\la),
	\end{split}
	\end{align}  
	with respect to the Lebesgue measure on $\RR^{d-1}\times\RR$, where $C\in (-\infty,1)$ is an absolute constant.
\end{lem}

\begin{figure}[t] 
	\centering
	\includegraphics[width=\textwidth]{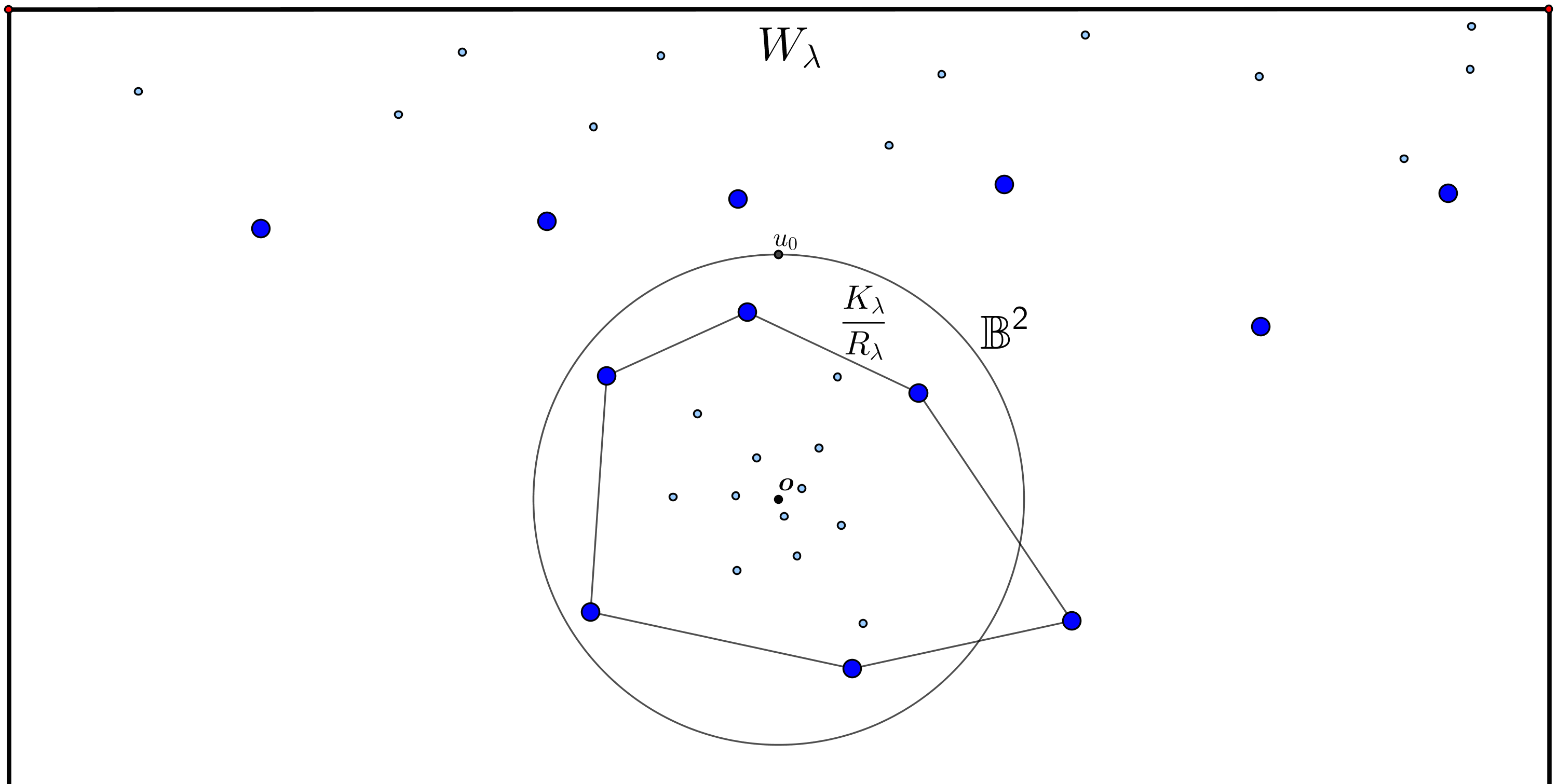}
	\caption{The rescaled Poisson point process $\mathcal{P}^{(\lambda)}$.}
	\label{Transformation..}
\end{figure}

Due to the properties of the sine function and the definition of $R_\la$, the first two fractions in \eqref{IntensityP^lambda} converge to $1$, as $\la \rightarrow \infty$, on compact subsets of $W_\la$. Moreover, for fixed $h \in \RR$, the same holds true for the fourth expression, while the exponential term tends to $e^h$, as $\la \rightarrow \infty$. Summarizing, this implies the following important corollary. 

\begin{cor}\label{CorrollarIntensityP^lambda}
	As $\la \rightarrow \infty$, $\cP^{(\la)}$ converges in distribution, in the sense of total variation convergence on compact sets, to the Poisson \index{$\mathcal{P}$} point process $\cP$ on $\RR^{d-1} \times \RR$, \textit{for all} parameter $\a$ and $\beta$. 
	%in the density $\phi_{\alpha,\beta}$.
\end{cor}

\begin{remark}
	The scaling transformation $T_\lambda$ carries $\cP_\lambda$ into the Poisson point process $\cP$ in the product space $\RR^{d-1} \times \RR$ that is stationary in the spatial coordinate, as $\lambda \rightarrow \infty$. On the one hand, this was to be expected in view of \cite[Theorem 4.1]{Eddy2}%(generalizing a result obtained in \cite{Eddy1})
	, where a transformation was constructed to carry the binomial counterpart of our $\mathcal{P}_\lambda$ into a point process in $\RR\times \RR^{d-1}$, whose height coordinate is determined by a Poisson point process with intensity $e^{-h}\dint h$, $h\in \RR$, while in the spatial regime a standard Gaussian point process arises. Thus, on the other hand, the result in \cite{Eddy2} clearly contrasts the latter corollary, in particular concerning the distribution in the spatial coordinate.
\end{remark}

\begin{proof}[Proof of Lemma \ref{LemmaIntensityP^lambda}]
	Let us write $x\in \RR^d$ as $x = ur$ with $u\in \mathbb{S}^{d-1}$ and $r\ge 0$. Thus, using polar coordinates, it follows that
	\begin{align*}
	\la\, \phi_{\a,\be}(x)\, \dint x = \la\, \phi_{\a,\be}(ur)\, r^{d-1}\, \dint r\, \cH_{\mathbb{S}^{d-1}}^{d-1}(\dint u),
	\end{align*}
	where $\cH_{\mathbb{S}^{d-1}}^{d-1}$ denotes the $(d-1)$-dimensional Hausdorff measure on $\SSd$.
	Following the proof of \cite[Lemma 3.2]{CalkaYukich}, we achieve, by making the change of variables 
	\begin{align*}
	v := R_\la^{\frac{\beta}{2}}\, \exp^{-1}(u) \qquad \text{and} \qquad h:= R_\la^\be\, \left(1 - \frac{r}{R_\la}\right)\, \Leftrightarrow\, r = R_\la \left(1 - \frac{h}{R_\la^\be}\right),
	\end{align*}
	that 
	\begin{align}\label{sigma_d-1}
	\cH_{\mathbb{S}^{d-1}}^{d-1}(\dint u) = \frac{\sin^{d-2} (R_\la^{-\frac{\beta}{2}}\|v\|)}{\|R_\la^{-\frac{\beta}{2}}v\|^{d-2}}\, \big(R_\la^{-\frac{\beta}{2}}\big)^{d-1}\, \dint v.
	\end{align}	
	%	Indeed, for all $v\in \RR^{d-1}\setminus \{\origin\}$, the exponential map can be expressed as 
	%	\begin{align*}
	%	\exp (v) = \cos(||v||) (\origin,1) + \sin(||v||) \left(\frac{v}{||v||}, 0\right),
	%	\end{align*}
	%	(see \cite[Equation (3.14)]{CalkaYukich}). Thus,
	%	\begin{align*}
	%	\cH_{\mathbb{S}^{d-1}}^{d-1}(\dint u) &= \sin(|| \exp^{-1}(u)||)^{d-2}\, \dint(||\exp^{-1}(u)||)\, \cH_{\mathbb{S}^{d-2}}^{d-2}\left(\dint \frac{ \exp^{-1}(u)}{|| \exp^{-1}(u)||}\right)\\ 
	%	&= \frac{\sin(|| \exp^{-1}(u)||)^{d-2}}{|| \exp^{-1}(u)||^{d-2}}\, \dint  (\exp^{-1}(u)),
	%	\end{align*}
	%	and the claim follows from $\exp^{-1}(u) = R_\la^{-\frac{\beta}{2}} v$. 
	Moreover, by the choice of $r$, 
	\begin{align}\label{dr}
	r^{d-1}\, \dint r = \left[R_\la\left(1 - \frac{h}{R_\la^\be}\right) \right]^{d-1}\, R_\la^{-(\be-1)}\, \dint h.
	\end{align}
	Furthermore, we get 
	\begin{align}\label{dlambda}
	\begin{split}
	\la\, \phi_{\a,\be}(ur) = (\be \log \la)^{\frac{\be(d+1) - 2d -2\a}{2\be}}\, R_\la^\a\, \left(1-\frac{h}{R_\la^\be}\right)^\a\, \exp\left(h - \frac{h^2}{2 R_\la^\be} (\be-1) (1-C)^{\be-2}
	\right),
	\end{split}
	\end{align}
	for some absolute constant $C\in (-\infty,1)$.
	Indeed, using the Taylor-Lagrange expansion up to second order of the function $(1-x)^\be$ at the point $0$ yields that 
	%there exists an absolute constant $C\in (0,x)$, satisfying
	%	\begin{align}\label{Lagrange}
	%	(1-x)^\be = 1 - \be x + \frac{x^2}{2} \be (\be-1) (1-C)^{\be-2}.
	%	\end{align} 
	%	The definitions of $r$ and $\phi_{\a,\be}(x)$, as well as \eqref{Lagrange} applied to $x=h/R_\la^\be$, imply that 
	there is an absolute constant $C\in (-\infty,1)$, satisfying 
	\begin{align*}
	\phi_{\a,\be}(ur) &= \phi_{\a,\be}\left(u\, R_\la\, \left(1-\frac{h}{R_\la^\be}\right) \right)\\
	&= c_{\a,\be}^d\, R_\la^\a\, \left(1-\frac{h}{R_\la^\be}\right)^\a\, \exp\left(-\frac{1}{\be}\, R_\la^\be\, \left(1-\frac{h}{R_\la^\be}\right)^\be\, \right)\\
	&= c_{\a,\be}^d\, R_\la^\a\, \left(1-\frac{h}{R_\la^\be}\right)^\a\, \exp\left(-\frac{1}{\be}\, R_\la^\be\, \left(1 - \be\frac{h}{R_\la^\be} + \frac{h^2}{2R_\la^{2\be}}\, \be (\be-1) (1-C)^{\be-2}\right)\, \right)\\
	&= c_{\a,\be}^d\, R_\la^\a\, \left(1-\frac{h}{R_\la^\be}\right)^\a\, \exp\left(-\frac{R_\la^\be}{\be} + h -  \frac{h^2}{2R_\la^{\be}}\, (\be-1) (1-C)^{\be-2} \right)\\
	&= c_{\a,\be}^d\, R_\la^\a\, \left(1-\frac{h}{R_\la^\be}\right)^\a\, \exp\left(-\frac{R_\la^\be}{\be}\right)\, \exp\left(h -  \frac{h^2}{2R_\la^{\be}}\, (\be-1) (1-C)^{\be-2}\right)\\
	&= \frac{1}{\la}\,R_\la^\a\, \left(1-\frac{h}{R_\la^\be}\right)^\a\, (\be \log \la)^{\frac{\be(d+1) - 2d - 2\a}{2\be}}\, \exp\left(h -  \frac{h^2}{2R_\la^{\be}}\, (\be-1) (1-C)^{\be-2} \right).
	\end{align*}
	Note that we used the explicit choice of $R_\la$ in the last step to deduce  
	\begin{align*}
	\exp\left(-\frac{R_\la^\be}{\be}\right) 
%	&= \exp\left(-\frac{1}{\be}\, \left(\be \log \la - \left(\frac{\be(d+1) - 2d - 2\a}{2}\right)\, \log \left(c_{\a,\be}^{- \frac{2\be d}{\be(d+1) - 2d - 2\a}}\, \be \log \la \right)\right)\right)\\
%	&= \exp(- \log \la)\, \exp\left(\left(\frac{\be(d+1) - 2d - 2\a}{2\be}\right)\, \log \left(c_{\a,\be}^{- \frac{2\be d}{\be(d+1) - 2d - 2\a}}\, \be \log \la \right)\right)\\
%	&= \frac{1}{\la}\, \left(c_{\a,\be}^{- \frac{2\be d}{\be(d+1) - 2d - 2\a}}\, \be \log \la \right)^{\frac{\be(d+1) - 2d - 2\a}{2\be}}\\
	&= \frac{1}{\la}\, (\be \log \la)^{\frac{\be(d+1) - 2d - 2\a}{2\be}}\, c_{\a,\be}^{-d}. 
	\end{align*}
	Combining \eqref{sigma_d-1}, \eqref{dr} and \eqref{dlambda} with
	\begin{align*}
	R_\lambda^\alpha\, R_\la^{-\frac{\be(d-1)}{2}}\, R_\la^{d-1}\, R_\la^{-(\be-1)} = R_\la^{\frac{-\be d + \be + 2d - 2 - 2\be + 2 + 2\alpha}{2}} = R_\la^{\frac{-\be(d+1) + 2d + 2\alpha}{2}}
	\end{align*}
	finishes the proof. 
\end{proof}

\subsection{Germ-grain processes}\label{sectiongermgrainmodels}

Following the notation introduced before Theorem 1.1, we define the unit upward paraboloid by
\begin{align*}
\Pi^{\uparrow}:=\left\{(v,h)\in\RR^{d-1}\times\RR:h\ge \frac{\|v\|^2}{2}\right\},
\end{align*}
giving rise to the limiting germ-grain process
\begin{align*}
\Psi:= \Psi(\cP) := \bigcup\limits_{w\in \mathcal{P}} [\Pi^\uparrow(w)]^{(\infty)}, 
\end{align*}
where $[\Pi^\uparrow(w)]^{(\infty)} := w\, \oplus\, \Pi^\uparrow$. Both $\Psi$ and $\Phi$ will play an important role in the subsequent analysis.
Let us continue this section with two observations regarding to the Generalized Gamma Polytope $K_\la$, derived and explained in detail for example in \cite[Page 14]{CalkaYukich} in the Gaussian case. First, a point $x'\in \cP_\la$ is a vertex of $K_\la$, if and only if the ball $\BB^d (\frac{x'}{2}, \frac{||x'||}{2})$
is not contained in the union of all balls corresponding to the other points of $\mathcal{P}_\lambda$, i.e, in $\bigcup_{y \in \cP_\la \atop y\neq x} \BB^d(\frac{y}{2}, \frac{||y||}{2})$.
We can rewrite such a ball as 
\begin{align}\label{Ball}
\begin{split}
\BB^d\left(\frac{x'}{2}, \frac{||x'||}{2}\right) %&= \{x\in \RR^d : ||x|| \le ||x'||\, \cos \theta \}\\
&= \left\{x\in \RR^d : R_\la^\be \left(1-\frac{||x||}{R_\la \cos \theta}\right) \ge R_\la^\be \left(1-\frac{||x'||}{R_\la}\right) \right\},
\end{split}
\end{align}
%where the second equality is achieved by an extension, 
where $\theta$
%\begin{align*}
%\theta := d_{\mathbb{S}^{d-1}}\left(\frac{x}{||x||}, \frac{x'}{||x'||}\right)
%\end{align*}
is the geodesic distance between $\frac{x}{||x||}$ and $\frac{x'}{||x'||}$ on the sphere. Secondly, $\RR^d \setminus K_\la$ is the union of half-spaces that do not contain points of $\cP_\la$. For $x'\in \RR^d$, consider the half-space
\begin{align}\label{Halfspace}
\begin{split}
H(x') :%&= \{x \in \RR^d : ||x'|| \le ||x||\, \cos \theta \} \\
& = \left\{x\in \RR^d : R_\la^\be \left(1-\frac{||x'||}{R_\la \cos \theta}\right) \ge R_\la^\be \left(1-\frac{||x||}{R_\la}\right) \right\},
\end{split}
\end{align} 
which is one of the main ingredients of the following lemma.

\begin{lem}\label{Germ-grain}
	Putting $T_\la(x') := (v',h')$, the scaling transformation $T_\la$ maps the ball $\BB^d(\frac{x'}{2}, \frac{||x'||}{2})$ and the half-space $H(x')$ into the upward opening grain 
	\begin{align}\label{upwardgrain}
	[\Pi^{\uparrow}(v',h')]^{(\lambda)} := \left\{(v,h)\in W_\la : h\geq R_\la^\be (1-\cos(d_\lambda(v',v))) + h' \cos(d_\lambda(v',v))\right\}\,
	\end{align}
	and the downward grain 
	\begin{align}\label{downwardgrain}
	[\Pi^{\downarrow}(v',h')]^{(\lambda)} := \left\{(v,h)\in W_\la : h\le R_\la^\be -\frac{R_\la^\be - h'}{\cos(d_\lambda(v',v))}\right\}\,,
	\end{align}
	respectively, where $d_\lambda(v',v)$
	%\begin{align*}
	%d_\lambda(v',v) := d_{\mathbb{S}^{d-1}}\left(\exp(R_\la^{-\frac{\beta}{2}} v'), \exp(R_\la^{-\frac{\beta}{2}}v)\right)
	%\end{align*}
	is the geodesic distance between images of rescaled points $v'$ and $v$ under the exponential map.
\end{lem}
\begin{proof}%[Proof of Lemma \ref{Germ-grain}]
	The characterization of the ball in \eqref{Ball} is equivalent to the inequality
	%\begin{align*}
	%R_\la^\be \left(1-\frac{||x||}{R_\la \cos \theta}\right) \ge R_\la^\be \left(1-\frac{||x'||}{R_\la}\right),
	%\end{align*}
	%if and only if
	\begin{align*}
	R_\la^\be \left(1-\frac{||x||}{R_\la}\right) \ge R_\la^\be (1-\cos \theta) + R_\la^\be \left(1-\frac{||x'||}{R_\la}\right) \cos \theta.
	\end{align*}
%	\begin{alignat*}{3}
%	& &&R_\la^\be \left(1-\frac{||x||}{R_\la \cos \theta}\right) &&\ge R_\la^\be \left(1-\frac{||x'||}{R_\la}\right)\\
%	%&\Leftrightarrow \qquad  &&R_\la^\be -\frac{R_\la^{\be-1}\, ||x||}{\cos \theta} &&\ge R_\la^\be \left(1-\frac{||x'||}{R_\la}\right)\\
%	%&\Leftrightarrow \qquad  &&R_\la^\be\, \cos \theta - R_\la^{\be-1}\, ||x|| &&\ge R_\la^\be \left(1-\frac{||x'||}{R_\la}\right) \cos \theta\\
%	%&\Leftrightarrow \qquad  &&R_\la^\be - R_\la^{\be-1}\, ||x|| &&\ge R_\la^\be - R_\la^\be\, \cos \theta + R_\la^\be \left(1-\frac{||x'||}{R_\la}\right) \cos \theta\\
%	&\Leftrightarrow \qquad  &&R_\la^\be \left(1-\frac{||x||}{R_\la}\right) &&\ge R_\la^\be (1-\cos \theta) + R_\la^\be \left(1-\frac{||x'||}{R_\la}\right) \cos \theta.
%	\end{alignat*} 
	Therefore,  
	\begin{align*}
	h \ge R_\la^\be (1-\cos (d_\lambda(v',v))) + h' \cos (d_\lambda(v',v)),
	\end{align*}
	where we used 
	\begin{align*}
	h' = R_\la^\be\, \left(1-\frac{||x'||}{R_\la}\right), \qquad  h &= R_\la^\be\, \left(1-\frac{||x||}{R_\la}\right), \qquad  v' = R_\la^{\frac{\beta}{2}}\, \exp^{-1}\left(\frac{x'}{||x'||}\right),
	\end{align*}
	and
	\begin{align*}
	v &= R_\la^{\frac{\beta}{2}}\, \exp^{-1}\left(\frac{x}{||x||}\right),
	\end{align*}
	in view of the scaling transformation $T_\la$. Similarly, we get from \eqref{Halfspace} that 
	\begin{align*}
	R_\la^\be \left(1-\frac{||x||}{R_\la}\right) &\le  R_\la^\be - \frac{R_\la^{\be}\, \frac{||x'||}{R_\la}}{\cos \theta} = R_\la^\be - \frac{R_\la^{\be}\, \left(1 - 1 + \frac{||x'||}{R_\la}\right)}{\cos \theta}
	&= R_\la^\be - \frac{R_\la^{\be} - R_\la^\be\left(1 - \frac{||x'||}{R_\la}\right)}{\cos \theta},
	\end{align*}
	and, thus,
	\begin{align*}
	h \le R_\la^\be -\frac{R_\la^\be - h'}{\cos(d_\lambda(v',v))}.
	\end{align*}
	This proves the claim. 
\end{proof}

Consequently, $T_\la$ transforms the sets 
\begin{align*}
\bigcup_{x\in \cP_\la} \BB^d\left(\frac{x}{2}, \frac{||x||}{2}\right) \qquad \text{and} \qquad  \RR^d \setminus K_\la 
\end{align*}
into the quasi-paraboloid \index{quasi-paraboloid germ-grain model} germ-grain \index{$\Psi^{(\lambda)}$} models
\begin{align*}
\Psi^{(\lambda)}:= \Psi^{(\lambda)}(T_\la(\cP_\la)) := \bigcup\limits_{w\in \mathcal{P}^{(\lambda)}} [\Pi^{\uparrow}(w)]^{(\lambda)}, 
\end{align*}
(see \index{$\Phi^{(\lambda)}$} Figure \ref{GermGrain1}), and
\begin{align*}
\Phi^{(\la)} := \Phi^{(\lambda)}(T_\la(\cP_\la)) := \bigcup_{w\in W_\la \atop \cP^{(\la)}\cap [\Pi^{\downarrow}(w)]^{(\lambda)}=\emptyset} [\Pi^{\downarrow}(w)]^{(\lambda)},
\end{align*}
(see Figure \ref{GermGrain2}), respectively.\\
%What is crucial about these germ-grain processes is that now, for sufficiently large $\la$, a point $x'\in \cP_\la$ is a vertex of $K_\la$, if and only if the germ $[\Pi^{\uparrow}(T_\la(x'))]^{(\la)}$ is not covered by $\Psi^{(\lambda)}(T_\la(\cP_\la \setminus \{x'\}))$. 
\begin{figure}[t] 
	\centering
	\includegraphics[width=\textwidth]{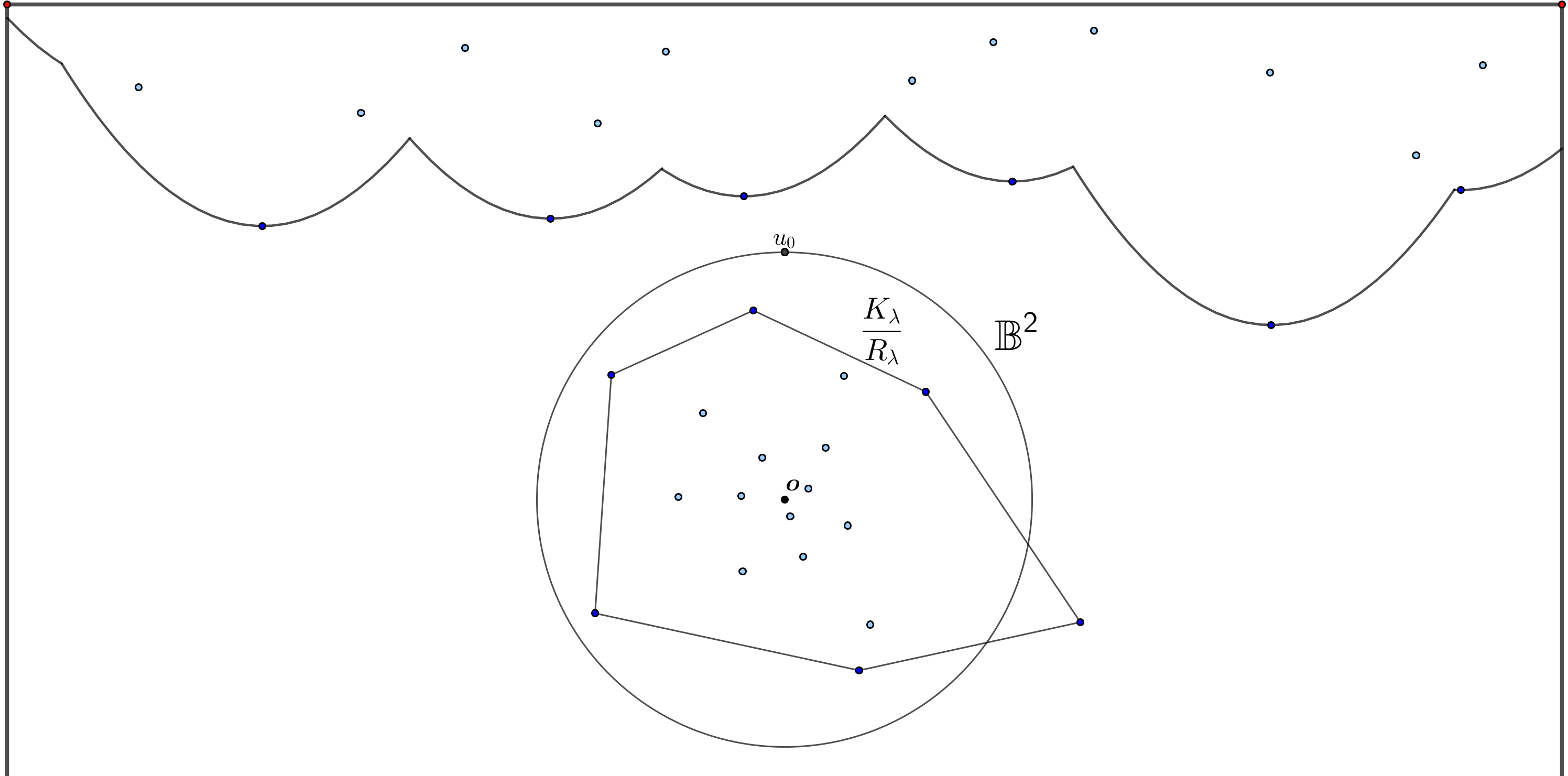}
	\caption{The germ-grain model $\Psi^{(\lambda)}$.}
	\label{GermGrain1}
\end{figure}
%This observation has been used extensively in the Gaussian case in \cite{CalkaYukich,GroteThäle2} and also our results exploit this fact. In this case, $T_\la(x')$ is called an extreme point of $\cP^{(\la)}$, whose collection we denote \index{$\ext(\cP^{(\la)}) $} by $\ext(\cP^{(\la)})$. Moreover, the boundary $\partial\Phi^{(\la)}$ of $\Phi^{(\la)}$ is build from piecewise quasi-parabolic facets, glued together at the extreme points of $\mathcal{P}^{(\la)}$.\\%, see again \cite{CalkaYukich}. 
%Moreover, for sufficiently large $\la$, $T_\la$ transforms the boundary of $K_\la$ into the upper boundary of $\Phi^{(\la)}$.
\begin{figure}[t] 
	\centering
	\includegraphics[width=\textwidth]{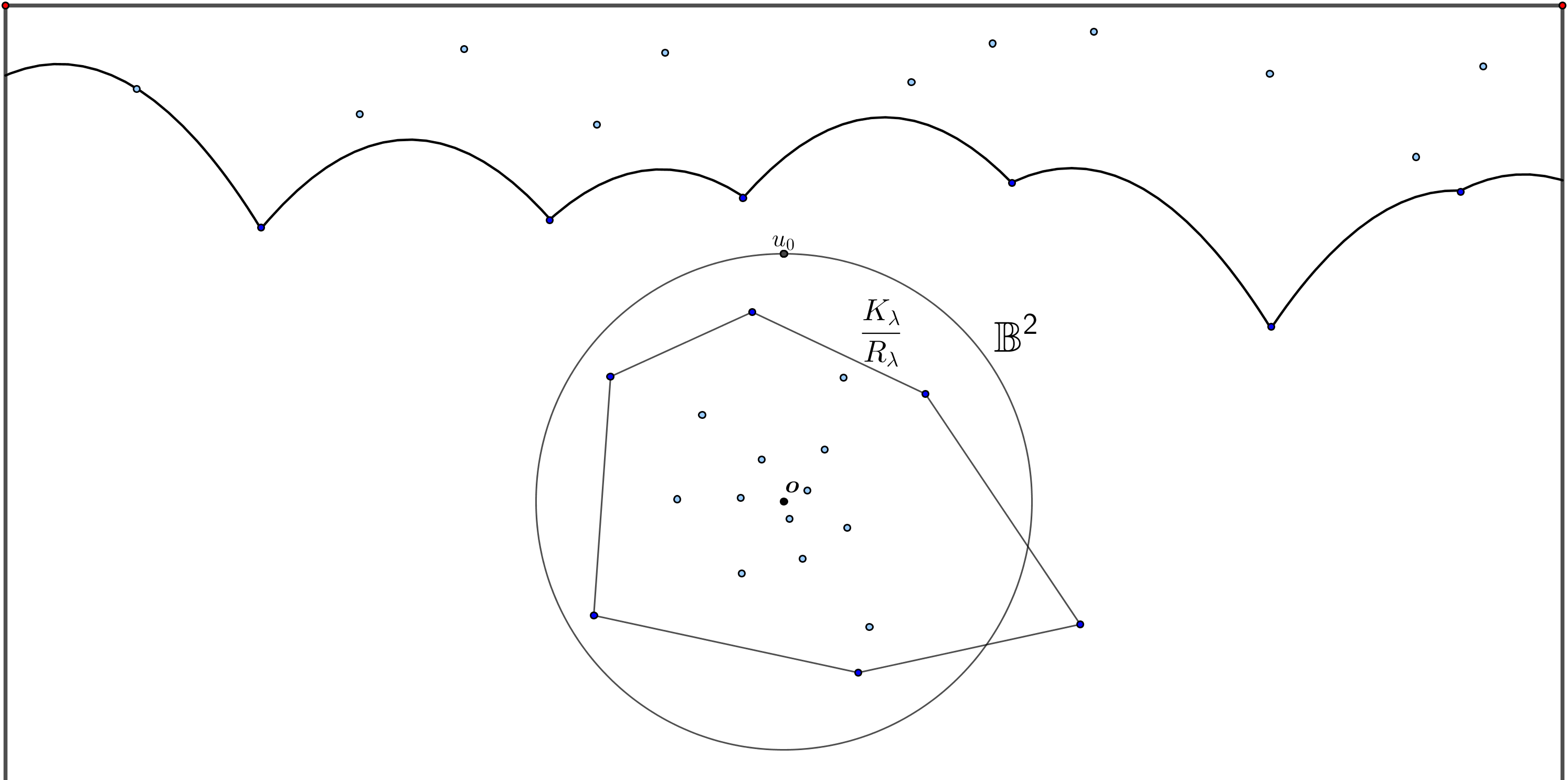}
	\caption{The germ-grain model $\Phi^{(\lambda)}$.}
	\label{GermGrain2}
\end{figure}

We continue with the following observation, a modification of \cite[Lemma 3.1]{CalkaYukich}. It shows that for fixed $w\in W_\lambda$, the quasi-paraboloids $[\Pi^\uparrow (w)]^{(\la)}$ and $[\Pi^\downarrow (w)]^{(\la)}$ locally approximate the paraboloids $[\Pi^\uparrow (w)]^{(\infty)}$ and $[\Pi^\downarrow (w)]^{(\infty)}$, respectively. Recall that $\BB^d(x,r)$ is the closed ball centered at $x\in \RR^d$ with radius $r>0$, and define the vertical cylinder $C_{d-1}(v,r)$ by $C_{d-1}(v,r) := \BB_{d-1}(v,r) \times \RR$. Moreover, $||\cdot||_{\infty}$ denotes the sup-norm of the argument function.

\begin{lem}\label{Bitch}
Let $w:=(v_1,h_1)\in W_\la$, $L\in (0,\infty)$ and $\la$ be sufficiently large. Then, it holds that
\begin{align*}
&||\partial ([\Pi^\uparrow (w)]^{(\la)} \cap C_{d-1}(v_1,L))- \partial ([\Pi^\uparrow (w)]^{(\infty)} \cap C_{d-1}(v_1,L))||_\infty\\ 
&\qquad \qquad \le C_1\, R_\la^{-\frac{1}{2}\be}\, L^3  + C_2\, h_1\, R_\la^{-\be}\, L^2,
\end{align*}
and 
\begin{align}\label{far}
\begin{split}
&||\partial ([\Pi^\downarrow (w)]^{(\la)} \cap C_{d-1}(v_1,L))- \partial ([\Pi^\downarrow (w)]^{(\infty)} \cap C_{d-1}(v_1,L))||_\infty\\
&\qquad \qquad \le C_3\, R_\la^{-\frac{1}{2}\be}\, L^3 + C_4\, h_1\, R_\la^{-\be}\, L^2,
\end{split}
\end{align}
where $C_1,C_2,C_3,C_4 \in (0,\infty)$ are absolute constants. %In particular, both right hand sides tend to 0, as $\lambda \rightarrow \infty$.
\end{lem}

\begin{proof}%[Proof of Lemma \ref{Bitch}]
We start with the first inequality and recall from \eqref{upwardgrain} that we have
\begin{align*}
\partial ([\Pi^{\uparrow}(w)]^{(\lambda)}) = \left\{(v,h)\in W_\la : h= R_\la^\be (1-\cos(d_\lambda(v_1,v))) + h_1 \cos(d_\lambda(v_1,v))\right\}.
\end{align*}
Let $v\in \BB_{d-1}(v_1,L)$. The Taylor expansion of the cosine function, together with
\begin{align}\label{preparation}
\begin{split}
d_\lambda(v_1,v) &= ||R_\la^{-\frac{\be}{2}}v - R_\la^{-\frac{\be}{2}}v_1|| + O(||R_\la^{-\frac{\be}{2}}v - R_\la^{-\frac{\be}{2}}v_1||^2)
%&=R_\la^{-\frac{\be}{2}}\, ||v - v_1|| + O(R_\la^{-\be}\, ||v - v_1||^2) 
=R_\la^{-\frac{\be}{2}}\, ||v - v_1|| + O(R_\la^{-\be}\, L^2),
\end{split}
\end{align}
gives
\begin{align*}
1 - \cos(d_\lambda(v_1,v)) &= \frac{d_\lambda(v_1,v)^2}{2} + O(d_\lambda(v_1,v)^3)
%&= R_\la^{-\be}\, \frac{||v - v_1||^2}{2} + O(R_\la^{-\frac{3}{2}\be}\, L^3) + O(R_\la^{-\frac{3}{2}\be}\, L^3) + O(R_\la^{-2\be}\, L^4)\\
= R_\la^{-\be}\, \frac{||v - v_1||^2}{2} + O(R_\la^{-\frac{3}{2}\be}\, L^3),
\end{align*}
as $\lambda \rightarrow \infty$. Thus,
\begin{align*}
R_\la^\be (1 - \cos(d_\lambda(v_1,v))) = \frac{||v - v_1||^2}{2} + O(R_\la^{-\frac{1}{2}\be}\, L^3),
\end{align*}
and
\begin{align*}
|h_1 (1 - \cos(d_\lambda(v_1,v)))| = O(h_1 R_\la^{-\be} L^2),
\end{align*}
as $\lambda \rightarrow \infty$.
The two last equations prove that the boundary of $[\Pi^\uparrow (w)]^{(\la)} \cap C_{d-1}(v_1,L)$ and the boundary of $[\Pi^\uparrow (w)]^{(\infty)} \cap C_{d-1}(v_1,L)$, which is given by the graph of 
\begin{align*}
v \mapsto h_1 + \frac{||v-v_1||^2}{2},
\end{align*}
(see the equations around \eqref{Limitprocess}), differ by at most $C_1 R_\la^{-\frac{1}{2}\be}\, L^3 + C_2 h_1 R_\la^{-\be} L^2$. This finishes the proof of the first assertion. 
Moreover, we have from \eqref{downwardgrain} that
\begin{align*}
\partial ([\Pi^{\downarrow}(w)]^{(\lambda)}) = \left\{(v,h)\in W_\la : h = R_\la^\be -\frac{R_\la^\be - h_1}{\cos(d_\lambda(v_1,v))}\right\}.
\end{align*}
By using again the Taylor expansion up to second order, the fact that 
\begin{align*}
\frac{1}{1-x} = 1 + x + x^2 + \ldots, 
\end{align*}
and the preparation \eqref{preparation}, we obtain for all $(v,h) \in \partial [\Pi^{\downarrow}(w)]^{(\lambda)} \cap C_{d-1}(v_1,L)$ that
\begin{align*}
h &= R_\la^\be -\frac{R_\la^\be - h_1}{\cos(d_\lambda(v_1,v))} = R_\la^\be -\frac{R_\la^\be - h_1}{\left(1 - \frac{d_\lambda(v_1,v)^2}{2}\right)}\\
&= R_\la^\be - (R_\la^\be - h_1)\left(1 + \frac{d_\lambda(v_1,v)^2}{2} + O(d_\lambda(v_1,v)^4) \right)\\
%&= R_\la^\be - (R_\la^\be - h_1)\left(1 + R_\la^{-\be}\, \frac{||v - v_1||^2}{2} + O(R_\la^{-\frac{3}{2}\be}\, L^3) + O(R_\la^{-2\be}\, L^4) \right)\\
&= R_\la^\be - (R_\la^\be - h_1)\left(1 + R_\la^{-\be}\, \frac{||v - v_1||^2}{2} + O(R_\la^{-\frac{3}{2}\be}\, L^3)\right)\\
&= R_\la^\be - R_\la^\be - \frac{||v - v_1||^2}{2} + O(R_\la^{-\frac{1}{2}\be}\, L^3) + h_1 + h_1 R_\la^{-\be}\, \frac{||v - v_1||^2}{2} + O(h_1 R_\la^{-\frac{3}{2}\be}\, L^3)\\
&= h_1 - \frac{||v - v_1||^2}{2} +  O(R_\la^{-\frac{1}{2}\be}\, L^3) + O(h_1 R_\la^{-\be}\, L^2),
\end{align*}
as $\lambda \rightarrow \infty$. Then, the result follows in the same way as in the first case. 
\end{proof}

In another crucial step in the proof of Theorem \ref{Festoon}, we prove that the boundaries of the germ-grain processes $\Psi^{(\lambda)}$, $\Phi^{(\lambda)}$, $\Psi$ and $\Phi$ do not only approximate each other, but are also `close' to the tangent plane $\RR^{d-1}$, with high probability. 
%The result in this direction reads as follows.  

\begin{thm}\label{WahrscheinlichkeitHöhe}
	For all $M\in (0,\infty)$, $t\ge 0$, $w:=(v,h) \in W_\la$, and sufficiently large $\la$, it holds that
	\begin{align*}
	\PP (\|\partial \Psi^{(\lambda)}(\mathcal{P}^{(\lambda)}) \cap C_{d-1}(v,M)\|_\infty \geq  t) \le c_1\, M^{2(d-1)}\exp\left(-c_2 t\right),
	\end{align*}
	and 
	\begin{align*}
	\PP (\|\partial \Psi(\mathcal{P}) \cap C_{d-1}(v,M)\|_\infty \geq  t) \le c_3\, M^{2(d-1)}\exp\left(-c_4 t\right),
	\end{align*}
	where $c_1,c_2,c_3,c_4 \in (0,\infty)$ are constants only depending on $d$, $\alpha$ and $\beta$.
	The two bounds also hold for the dual processes $\Phi^{(\la)}$ and $\Phi$.
\end{thm}

\begin{remark}
	As aforementioned and proven in Corollary \ref{CorrollarIntensityP^lambda}, the limiting Poisson point process $\cP$, as well as the corresponding germ-grain models $\Psi$ and $\Phi$, do \textit{not} depend on the parameter $\a$ and $\be$ from the underlying distribution. Hence, the proofs of the assertions for these three limiting processes stated in Theorem \ref{WahrscheinlichkeitHöhe} stay literally the same compared with the ones derived in the Gaussian case in \cite{CalkaYukich}, and can therefore be omitted. Thus, it remains to derive the above stated assertions connected with $\cP^{(\la)}$, $\Psi^{(\la)}$ and $\Phi^{(\la)}$, which depend on $\a$ and $\be$ by definition.
\end{remark}

Due to the rotational invariance of the underlying Poisson point process $\cP_\la$, it is enough to prove Theorem \ref{WahrscheinlichkeitHöhe} for points $w=(\origin,h)\in W_\la$ with $h\in (-\infty,R_\la^\be]$. Let $M\in (0,\infty)$, $t\ge 0$, $\la$ be sufficiently large, and define the events 
\begin{align*}
T_1 := \{\partial \Psi^{(\la)}(\cP^{(\la)}) \cap \{(v,h) : ||v|| \le M, h>t\} \ne \emptyset\},
\end{align*}
and 
\begin{align*}
T_2 := \{\partial \Psi^{(\la)}(\cP^{(\la)}) \cap \{(v,h) : ||v|| \le M, h< -t\} \ne \emptyset\}.
\end{align*}

We show the following two estimates, leading to the proof of Theorem \ref{WahrscheinlichkeitHöhe}.

\begin{lem}\label{T_1}
	For sufficiently large $\la$, it holds that
	\begin{align*}
	\PP(T_1) \le c_1\, M^{d-1}\, \exp(-c_2 e^t) \qquad \text{and} \qquad 
	\PP(T_2) \le c_3\, M^{2(d-1)}\, \exp(-c_4 t),
	\end{align*}
	where $c_1,c_2,c_3,c_4 \in (0,\infty)$ are constants only depending on $d$, $\alpha$ and $\beta$.
\end{lem}

\begin{proof}[Proof of Theorem \ref{WahrscheinlichkeitHöhe}]
	Recalling the definition of the events $T_1$ and $T_2$ in combination with the results from Lemma \ref{T_1} gives that
	\begin{align*}
	\PP (\|\partial \Psi^{(\lambda)}(\mathcal{P}^{(\lambda)}) \cap C_{d-1}(v,M)\|_\infty \geq  t) = \PP(T_1) + \PP(T_2)
	%&\le c_1\, M^{d-1}\, \exp(-c_2 e^t) + c_3\, M^{2(d-1)}\, e^{-c_4 t}\\
	\le c_1\, M^{2(d-1)}\, \exp(-c_2 t),
	\end{align*}
	where $c_1,c_2 \in (0,\infty)$ are constants only depending on $d$, $\alpha$ and $\beta$. This finishes the proof. 
\end{proof}

Thus, it remains to prove Lemma \ref{T_1}, and we start with the first assertion. Similarly to what has been done in \cite[Page 25]{CalkaYukich}, the event $T_1$ can be rewritten in the form 
	\begin{align*}
	T_1 = \{\exists w_1:=(v_1,h_1) \in \partial \Psi^{(\la)}(\cP^{(\la)}) : h_1\ge t, ||v_1|| \le M, [\Pi^\downarrow(w_1)]^{(\la)} \cap \cP^{(\la)} = \emptyset \},
	\end{align*}
	(see Figure \ref{figure3}).
	\begin{figure}[t] 
		\centering
		\includegraphics[width=\textwidth]{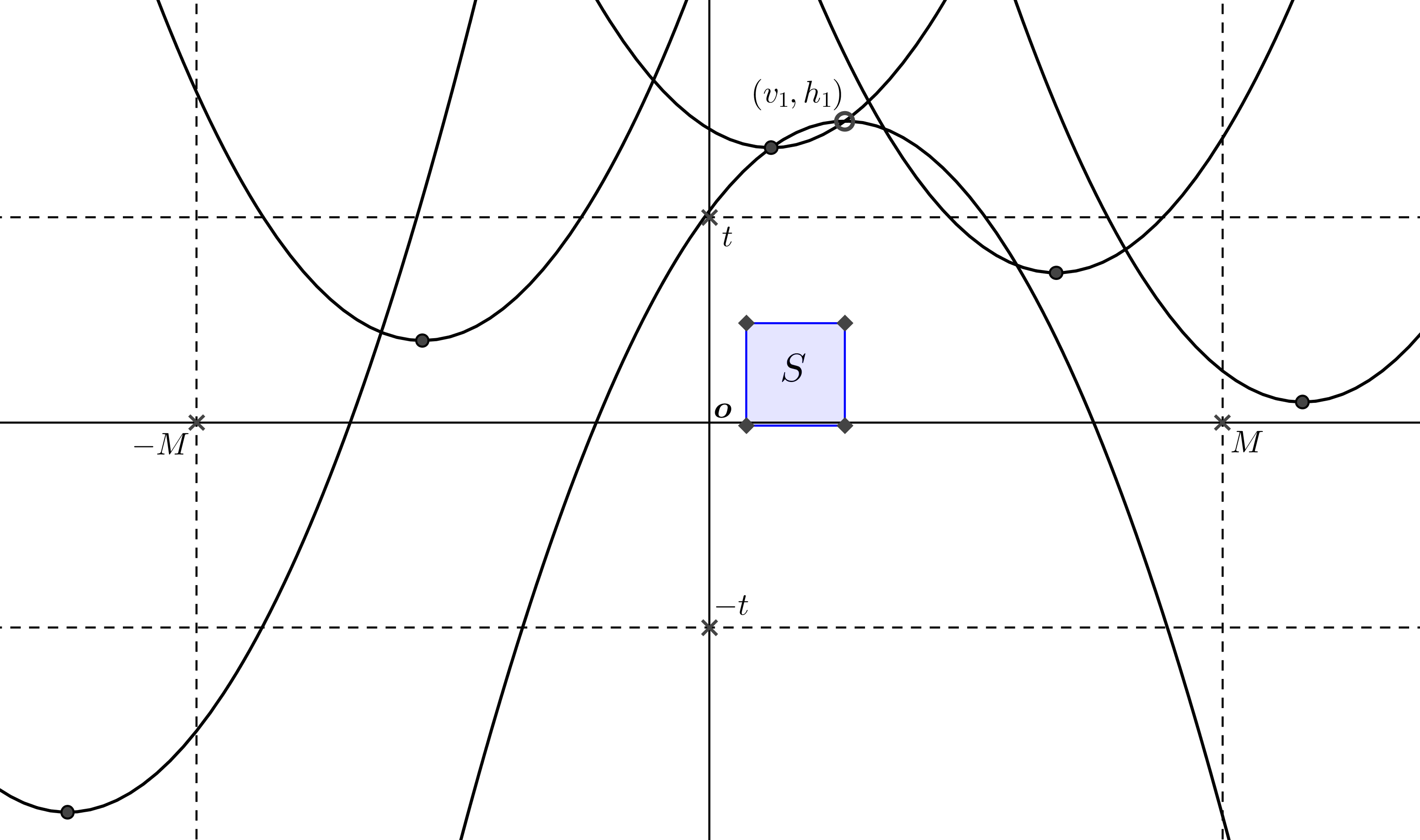}
		\caption{The event $T_1$ and the unit volume cube $S$.}
		\label{figure3}
	\end{figure}
	%The only difference is that there is now also a condition on the spatial coordinate $v_1$. 
	Fix $w_1:=(v_1,h_1) \in \partial \Psi^{(\la)}(\cP^{(\la)})$ and define the inverse of the scaling transformation of $w$ by $\rho u_0 :=T_\la^{-1}(w)$, $\rho > 0$, where we recall that $u_0$ indicates the north pole on the sphere $\mathbb{S}^{d-1}$. The parameter $\rho$ is positive, since otherwise the spatial coordinate of $w$ would be $\pi R_\la^{\be/ 2}$ instead of $\origin$, by definition of $T_\la$.  
	
	\begin{lem}\label{CubeS}
	Denote by $S$ the unit volume cube centered in 
	\begin{align*}
	\left(v_1 - \frac{\sqrt{d-1} v_1}{2 ||v_1||}, \frac{h_1}{(\be+1)^\be} - 1 \right),
	\end{align*}
	(see Figure \ref{figure3}). For sufficiently large $\la$, it fulfills
	\begin{align}\label{unitcubeS}
	S \subseteq [\Pi^\downarrow(w_1)]^{(\la)} \cap C_{d-1}\Big(\origin, M \wedge \frac{3 \pi R_\la^{\be \over 2}}{4}\Big),
	\end{align} 
	where $a\wedge b$ denotes the minimum of $a,b\in \RR$.
    \end{lem}
    
    \begin{proof}%[Proof of Lemma \ref{CubeS}]
	Due to the estimate in \eqref{far}, the boundaries of $[\Pi^\downarrow(w_1)]^{(\la)}$ and $[\Pi^\downarrow(w_1)]^{(\infty)}$ are not `far' from each other, and the latter downward germ contains the cube $S$ by definition, showing $S \subseteq [\Pi^\downarrow(w_1)]^{(\la)}$, for sufficiently large $\la$. Furthermore, the ball $\BB^d(\frac{\rho u_0}{2},\frac{\rho}{2})$, that is mapped into the germ $[\Pi^{\uparrow}(w)]^{(\lambda)}$ by the scaling transformation $T_\la$ (see Lemma \ref{Germ-grain}), is a subset of $\RR^{d-1} \times (0,\infty)$, since $\rho > 0$. Additionally, $T_\la$ transforms this upper half space into the cylinder $C_{d-1}(\origin, \pi R_\la^{\be/ 2}/2)$.
	This leads to the relation
	\begin{align*}
	[\Pi^{\uparrow}(w)]^{(\lambda)} = T_\la\left(\BB^d\left(\frac{\rho u_0}{2},\frac{\rho}{2}\right)\right) \subseteq C_{d-1}\Big(\origin,\frac{\pi R_\la^{\be \over 2}}{2}\Big),
	\end{align*}
	which implies $||v_1|| \le \pi R_\la^{\be/2}/2$ and, therefore, $S \subseteq C_{d-1}\Big(\origin,3 \pi R_\la^{\be/2}/4\Big)$. The shift in the spatial coordinate of the center of $S$ is necessary to ensure that also $S \subseteq C_{d-1}(\origin,M)$, proving the lemma.
    \end{proof}
	
	The cube $S$ is the main ingredient when proving the next assertion.  
	
	\begin{lem}\label{LemmaS}
		For sufficiently large $\la$, it holds that 
		\begin{align*}
		\PP ([\Pi^\downarrow(w_1)]^{(\la)} \cap \cP^{(\la)} = \emptyset)  \le \exp(-c_1\, e^{c_2 h_1}),
		\end{align*}
		where $c_1,c_2 \in (0,\infty)$ are constants only depending on $d$, $\alpha$ and $\beta$.
	\end{lem}
	
	\begin{proof}%[Proof of Lemma \ref{LemmaS}]
		Let $(v,h) \in S$. From the definition of the cube $S$, we get that 
		%\begin{align*}
		%h \in \left[\frac{h_1}{(\be+1)^\be} - \frac{3}{2}, \frac{h_1}{(\be+1)^\be} - \frac{1}{2}\right],
		%\end{align*}
		%and, thus,
		\begin{align}\label{h/R_l}
		\begin{split}
		\frac{h}{R_\la^\be}  &\in \left[\frac{h_1}{(\be+1)^\be R_\la^\be} - \frac{3}{2R_\la^\be}, \frac{h_1}{(\be+1)^\be R_\la^\be} - \frac{1}{2R_\la^\be}\right] %\subseteq \left[-\frac{3}{2}, \frac{1}{(\be+1)^\be}\right] 
		\subseteq \left[-\frac{3}{2}, \frac{1}{2}\right],
		\end{split}
		\end{align}
		since $h_1/R_\la^\be \in [0,1]$ and $\beta \ge 1$. Hence, in view of \eqref{IntensityP^lambda}, there is some $C\in \left[-\frac{3}{2}, \frac{1}{2}\right]$ such that the density of the intensity measure of $\cP^{(\la)}$ in each point $(v,h) \in S$ can be expressed as 
		\begin{small}
		\begin{align}\label{hier}
		\begin{split}
		&\frac{\sin^{d-2} (R_\la^{-\frac{\beta}{2}}\|v\|)}{\|R_\la^{-\frac{\beta}{2}}v\|^{d-2}}\, \frac{(\be\log \la)^{\frac{\be(d+1) - 2d - 2\a}{2\be}}}{R_\la^{\frac{\be(d+1) - 2d - 2\a}{2}}} \exp\left(h - \frac{h^2}{2 R_\la^\be} (\be-1) (1-C)^{\be-2}\right)
		\left(1-\frac{h}{R_\la^\be}\right)^{d-1+\a}.
		\end{split}
		\end{align}  
	    \end{small}
		Besides, the preparation \eqref{unitcubeS} implies that 
		\begin{align*}
		R_\la^{-\frac{\beta}{2}}\|v\| \le R_\la^{-\frac{\beta}{2}}\, \frac{3 \pi R_\la^{\frac{\beta}{2}}}{4} = \frac{3 \pi}{4}.
		\end{align*}
		Therefore, for sufficiently large $\lambda$, the first fraction is bounded from below by a positive constant. Moreover, if the exponent $\frac{\be(d+1) - 2d - 2\a}{2\be}$ is positive, the definition of $R_\la$ yields that 
		\begin{align*}
		&\frac{(\be\log \la)^{\frac{\be(d+1) - 2d - 2\a}{2\be}}}{R_\la^{\frac{\be(d+1) - 2d - 2\a}{2}}}
		%\\&\qquad = \Bigg(\frac{\be \log \la}{\be \log \la - \underbrace{\Big(\frac{\be(d+1) - 2d - 2\a}{2}\Big)}_{> 0} \log \Big(c_{\a,\be}^{- \frac{2\be d}{\be(d+1) - 2d - 2\a}} \be \log \la \Big)}\Bigg)^{\frac{\be(d+1) - 2d - 2\a}{2\be}} 
		> 1.
		\end{align*}
		%since the term in the bracket is larger than 1. 
		If the exponent is negative, we achieve the same bound by definition of $R_\la$. 
		%	\begin{align*}
		%	&\frac{(\be\log \la)^{\frac{\be(d+1) - 2d - 2\a}{2\be}}}{R_\la^{\frac{\be(d+1) - 2d - 2\a}{2}}}\\
		%	&\qquad = \Bigg(\frac{\be \log \la}{\be \log \la - \underbrace{\Big(\frac{\be(d+1) - 2d - 2\a}{2}\Big)}_{< 0} \log \Big(c_{\a,\be}^{- \frac{2\be d}{\be(d+1) - 2d - 2\a}} \be \log \la \Big)}\Bigg)^{\frac{\be(d+1) - 2d - 2\a}{2\be}} > 1.
		%	\end{align*}
		%	Here, the inner fraction is smaller than 1, but since we have a negative exponent, we nevertheless achieve the statement. 
		Summarizing, the second fraction in \eqref{hier} is larger than 1. 
		Let us switch to the height coordinate $h$. First, notice that $d-1+\a > 0$, since $\a > -1$. If $h\le 0$, the fourth term in \eqref{hier} is larger than $1$. In the other case, the estimate derived in \eqref{h/R_l} yields that
		\begin{align*}
		\left(1-\frac{h}{R_\la^\be}\right)^{d-1+\a} \ge \left(\frac{1}{2}\right)^{d-1+\a} > 0.
		\end{align*}
		Moreover, the third expression in \eqref{hier} is bounded from below by $c_1 \exp(c_2 h_1)$, where $c_1,c_2 \in (0,\infty)$ are constants only depending on $d$, $\alpha$ and $\beta$. Indeed, we have
		\begin{align*}
		h\in \left[\frac{h_1}{(\be+1)^\be} - \frac{3}{2}, \frac{h_1}{(\be+1)^\be} - \frac{1}{2}\right] \subseteq \left[\frac{h_1}{(\be+1)^\be} - \frac{3}{2}, \frac{h_1}{(\be+1)^\be}\right].
		\end{align*}
		On these grounds, since $h_1/R_\la^\be\in [0,1]$, 
		\begin{align*}
		\exp\Big(h - \frac{h^2}{2 R_\la^\be} (\be-1) (1-C)^{\be-2}\Big)
		%&\ge \exp\left(-\frac{3}{2}\right) \exp\left(\frac{h_1}{(\beta +1)^\beta}   - \frac{h_1^2}{2 (\be+1)^{2\be} R_\la^\be} (\be-1) (1-C)^{\be-2}\right)\\
		&\ge \exp\Big(-\frac{3}{2}\Big) \exp\Big(\frac{h_1}{(\beta +1)^\beta}   - \frac{h_1}{2 (\be+1)^{2\be}} (\be-1) (1-C)^{\be-2}\Big)\\
		&= \exp\left(-\frac{3}{2}\right) \exp\left(h_1  \frac{2(\beta +1)^\beta - (\be-1) (1-C)^{\be-2}}{2 (\be+1)^{2\be}}\right)\\
		&\ge \exp\Big(-\frac{3}{2}\Big) \exp\Big(h_1  \frac{(\beta +1)^\beta}{2 (\be+1)^{2\be}}\Big)\\
		&= \exp\Big(-\frac{3}{2}\Big) \exp\Big(\frac{h_1}{2 (\be+1)^{\be}}\Big),
		\end{align*}
		where in the last inequality we have used that
		\begin{align*}
		2(\be+1)^\be - (\be-1) (1-C)^{\be-2} \ge (\be + 1)^\be,
		\end{align*}
		since $C \in \left[- \frac{3}{2}, \frac{1}{2}\right]$. This proves the claim.\\
		Summarizing the last calculations, we obtain that the density of the intensity measure of $\cP^{(\la)}$, evaluated in an arbitrary point $(v,h)\in S$, can be bounded from below by $c_1 \exp(c_2 h_1)$.
		Since the cube $S$ has by construction unit volume, we obtain, writing $\nu_\la$ for the intensity measure of the rescaled Poisson point process $\cP^{(\la)}$, that 
		\begin{align*}
		\nu_\la(S) \ge c_1 \exp(c_2  h_1).
		\end{align*}
		Therefore,
		\begin{align*}
		\PP ([\Pi^\downarrow(w_1)]^{(\la)} \cap \cP^{(\la)} = \emptyset) =  \exp(-\nu_\la ([\Pi^\downarrow(w_1)]^{(\la)})) \le \exp(-\nu_\la (S)) \le \exp(-c_1\, e^{c_2 h_1}),
		\end{align*}
		where $c_1,c_2 \in (0,\infty)$ are constants only depending on $d$, $\alpha$ and $\beta$. This completes the proof.
	\end{proof}

    %Finally, we are able to prove Lemma \ref{T_1}.

    \begin{proof}[Proof of Lemma \ref{T_1}]
    Since the Euclidean norm of the spatial
    	\begin{figure}[t] 
    		\centering
    		\includegraphics[width=\textwidth]{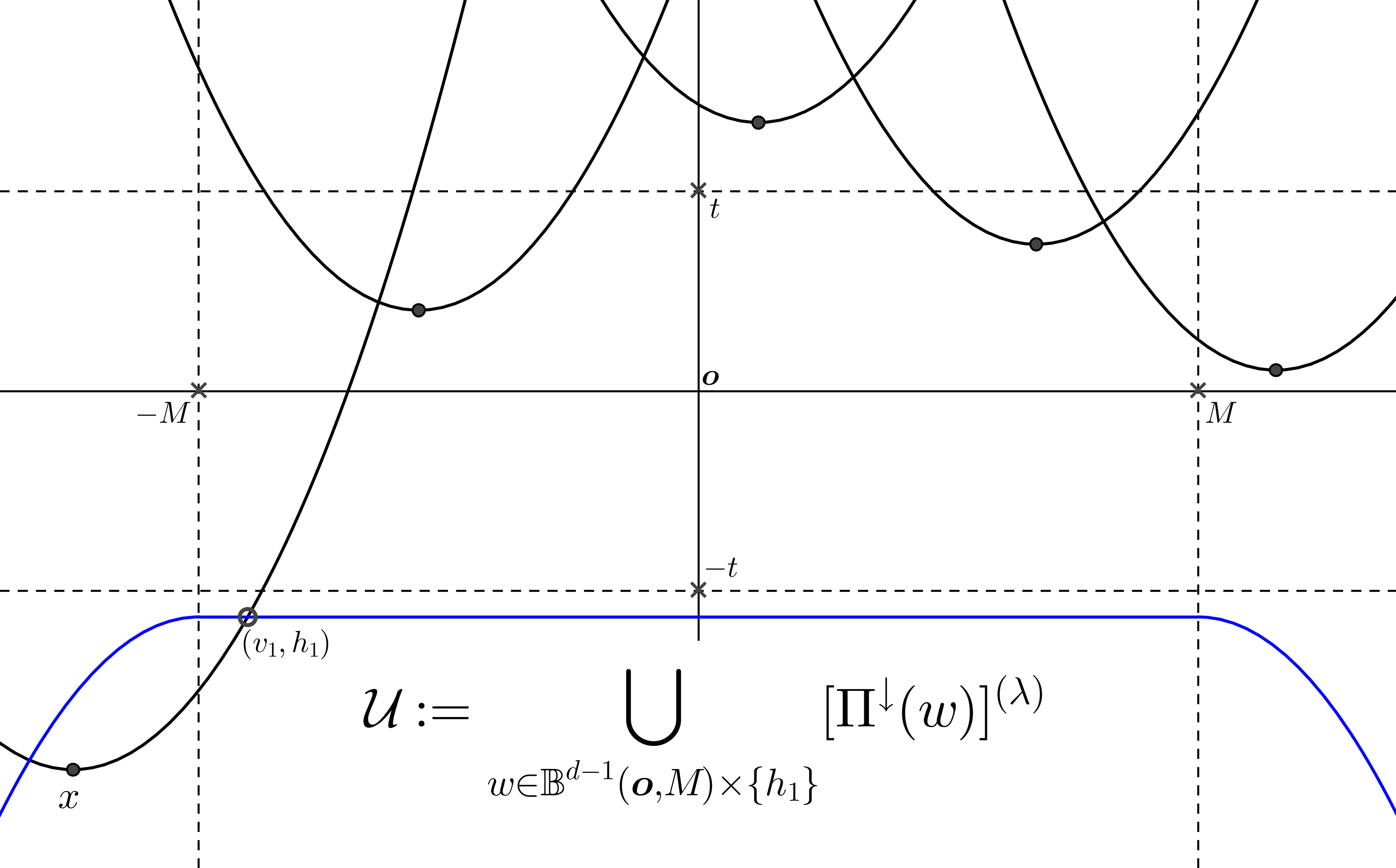}
    		\caption{The set $\mathcal{U}$ and the event $T_2$.}
    		\label{figure4}
    	\end{figure}	
    coordinate of $w_1$ is bounded by $M$ and the height coordinate is larger than $t$, we get, for sufficiently large $\lambda$,
	\begin{align*}
	\PP(T_1) \le c_1 \int\limits_t^\infty M^{d-1}\, \PP ([\Pi^\downarrow(w_1)]^{(\la)} \cap \cP^{(\la)} = \emptyset) \, \dint h_1
	%&\le c_1\, M^{d-1} \int\limits_t^\infty \exp(-c_2\, e^{c_3 h_1})\\
	\le c_2\, M^{d-1}\, \exp(-c_3 e^{c_4t}),
	\end{align*}
	where $c_1,c_2,c_3,c_4 \in (0,\infty)$ are constants only depending on $d$, $\alpha$ and $\beta$, and we used Lemma \ref{LemmaS} in the last step. This completes the proof for the event $T_1$, and we switch to $T_2$. If $T_2$ occurs, then, there must be an explicit point $x\in \cP^{(\la)}$ with
	\begin{align*}
	T_2 = \{\exists w_1:=(v_1,h_1) \in \partial \Psi^{(\la)}(x) : h_1\in (-\infty, -t], ||v_1|| \le M \},
	\end{align*}
	and 
	\begin{align*}
	x\in \mathcal{U} := \bigcup\limits_{w\in \BB^{d-1}(\origin,M) \times \{h_1\} }\, [\Pi^{\downarrow}(w)]^{(\lambda)},
	\end{align*}
	illustrated by Figure \ref{figure4} in the plane. 
	Writing again $\nu_\la$ for the intensity measure of the rescaled Poisson point process $\cP^{(\la)}$, by using \eqref{IntensityP^lambda}, we obtain that
	\begin{align}\label{Bound1}
	\begin{split}
	\nu_\la(U)&= \int\limits_{-\infty}^{h_1}\, \int\limits_{\BB_{d-1}(\origin,v); \atop (v,h_1) \in \mathcal{U}}\, \frac{\sin^{d-2} (R_\la^{-\frac{\beta}{2}}\|v'\|)}{\|R_\la^{-\frac{\beta}{2}}v'\|^{d-2}}\, \frac{(\be\log \la)^{\frac{\be(d+1) - 2d - 2\a}{2\be}}}{R_\la^{\frac{\be(d+1) - 2d - 2\a}{2}}}\\
	&\qquad \times \exp\left(h - \frac{h^2}{2 R_\la^\be} (\be-1) (1-C)^{\be-2}\right)
	\left(1-\frac{h}{R_\la^\be}\right)^{d-1+\a}\,  \dint v'\, \dint h,
	\end{split}
	\end{align}
	for some absolute $C\in (-\infty,1)$. Now, for sufficiently large $\lambda$, the two fractions and the exponential term are bounded from above by 1, a positive constant and $e^{h'}$, respectively, since $\beta \ge 1$. By using the fact that the spatial region is bounded by $M$, we get similarly as before that
	\begin{align}\label{Bound2}
	\begin{split}
	\nu_\la(\mathcal{U})&\le c_1\, \int\limits_{-\infty}^{h_1}\, M^{d-1}\, e^{h}\, \left(1-\frac{h}{R_\la^\beta}\right)^{d-1+\a}\, \dint h 
	%= c\, M^{d-1} \int\limits_{-h_1}^\infty e^{-h}\, \left(1+\frac{h}{R_\la^\beta}\right)^{d-1+\a}\, \dint h\\
	= c_2\, M^{d-1}\, \exp(c_3 h_1),
	\end{split}
	\end{align}
	where $c_1,c_2,c_3 \in (0,\infty)$ are constants only depending on $d$, $\alpha$ and $\beta$.
	%	where we used $0 \le -h_1$ and, thus, 
	%	\begin{align*}
	%	e^{-h}\, \left(1+\frac{h}{R_\la^\beta}\right)^{d-1+\a} \le e^{-ch}.
	%	\end{align*}
	This implies that
	\begin{align*}
	\PP(\mathcal{U} \cap \cP^{(\la)} \ne \emptyset) = 1 - \PP(\mathcal{U} \cap \cP^{(\la)} = \emptyset) = 1 - \exp(- \nu_\la(\mathcal{U})) \le \nu_\la(\mathcal{U}) \le c_1\, M^{d-1}\, \exp(c_2 h_1).
	\end{align*}
	Finally, this yields
	\begin{align*}
	\PP(T_2) \le  c_1\, \int\limits_{-\infty}^{-t} M^{d-1}\, \PP(\mathcal{U} \cap \cP^{(\la)} \ne \emptyset)\, \dint h_1 \le c_2\, M^{2(d-1)}\,  \int\limits_{-\infty}^{-t} e^{c_3 h_1}\, \dint h_1 = c_4\, M^{2(d-1)}\, \exp(-c_5 t),
	\end{align*}
	where $c_1,c_2,c_3,c_4,c_5 \in (0,\infty)$ are constants only depending on $d$, $\alpha$ and $\beta$. This finishes the proof of the lemma.
\end{proof}

We close Section 2 by stating the final steps in the proof of Theorem \ref{Festoon}, following along the same lines as the proof of \cite[Propostion 5.1]{CalkaYukich}.
%Based on all these preparations, we are finally able to prove the main result. 

\subsection{The final step in the proof of Theorem 1.1}

\begin{proof}[Proof of Theorem \ref{Festoon}]
Instead of proving the two results stated in the main theorem directly, we show an even stronger result, namely, that for fixed $L\in (0,\infty)$, the boundary of the germ-grain process $\Psi^{(\lambda)}(\mathcal{P}^{(\lambda)})$ converges in probability to the boundary of the limiting germ-grain process $\Psi(\mathcal{P})$, as $\la \rightarrow \infty$, in the space $C_{d-1}(\BB_{d-1}(\origin,L))$, equipped with the supremum norm. A similar statement holds for the boundaries of $\Phi^{(\lambda)}(\mathcal{P}^{(\lambda)})$ and $\Phi(\mathcal{P})$. These results contain the one from part $(b)$ and imply the one from $(a)$ of Theorem \ref{Festoon}, respectively. As before, we focus on the process $\Psi$, stressing that the proof for $\Phi$ is similar.\\
Let $L\in (0,\infty)$ be fixed. Then, for $\ell,\la \ge 0$, we denote by $E:=E(L,\ell,\la)$ the event that for all points in $\BB_{d-1}(\origin,L)$, the corresponding heights of $\partial \Psi^{(\lambda)}(\mathcal{P}^{(\lambda)})$, as well as $\partial \Psi(\mathcal{P})$, belong to the set $[-\ell,\ell]$. Now, the crucial Theorem \ref{WahrscheinlichkeitHöhe} implies that 
\begin{align*}
\PP(E^c) \le c_1\, L^{2(d-1)}\exp\left(-c_2 \ell\right),
\end{align*}
where $c_1,c_2 \in (0,\infty)$ are constants only depending on $d$, $\alpha$ and $\beta$. Thus, it remains to show that for sufficiently large $\la$, the boundary of $\Psi^{(\lambda)}(\mathcal{P}^{(\lambda)})$ is `close' to the boundary of $\Psi(\mathcal{P})$, conditioned on the event $E$. Therefore, it is sufficient to show that the boundaries of
%\begin{small}
\begin{align}\label{66}
\bigcup\limits_{w\in \mathcal{P}^{(\lambda)}\, \cap\, C_{d-1}(\origin,L)} ([\Pi^{\uparrow}(w)]^{(\lambda)} \cap C_{d-1}(\origin,L)) \quad \text{and} \quad \bigcup\limits_{w\in \mathcal{P}\, \cap\, C_{d-1}(\origin,L)} ([\Pi^{\uparrow}(w)] \cap C_{d-1}(\origin,L))
\end{align}
%\end{small}
are `close' to each other, again conditioned on $E$. Given some $w_1:=(v_1,h_1)\in \mathcal{P}^{(\lambda)} \cap C_{d-1}(\origin,L)$, we know from Lemma \ref{Bitch} that, conditioned on $E$, the boundary of $[\Pi^{\uparrow}(w_1)]^{(\lambda)}\, \cap\, C_{d-1}(\origin,L)$ is within $O(R_\la^{-\beta/2})$ of the one of $[\Pi^{\uparrow}(w_1)]\, \cap\, C_{d-1}(\origin,L)$. Since the boundary of $\Psi^{(\lambda)}(\mathcal{P}^{(\lambda)})\, \cap\, C_{d-1}(\origin,L)$ is built almost surely by a finite union of graphs of the above form, it is also almost surely within $O(R_\la^{-\beta/2})$ of the boundary of 
\begin{align}\label{77}
\bigcup\limits_{w\in \mathcal{P}^{(\lambda)}\, \cap\, C_{d-1}(\origin,L)} ([\Pi^{\uparrow}(w)] \cap C_{d-1}(\origin,L)).
\end{align}
Thus, it suffices to show that the boundary of the second process from \eqref{66} is `close' to the boundary of the one in \eqref{77}. In order to achieve this, we may construct some coupling between $\mathcal{P}^{(\la)}$ and $\mathcal{P}$ on the set $\BB_{d-1}(\origin,L) \times [-\ell,\ell]$. After this coupling, the two latter mentioned germ-grain processes coincide, except on a set that has probability less than $\varepsilon$, for some $\varepsilon > 0$. This proves the desired statement with a probability at least $1-\varepsilon$, showing the claim. 
%The rest of the proof is exactly the same as the proof of \cite[Proposition 5.1]{CalkaYukich} in the Gaussian case. It just uses the two bounds stated above, in combination with Theorem \ref{WahrscheinlichkeitHöhe} from Section \ref{sec:stabilization}. As in the proof of Lemma \ref{LemmaKonvergenz}, we omit stating a word by word repetition from \cite[Page 34]{CalkaYukich}.\\		
%The proof is word by word the same as the proof of \cite[Proposition 5.1]{CalkaYukich} in the Gaussian case. It uses the latter lemma and the results from Section \ref{sec:stabilization} and is, therefore, again completely independent of the parameter $\alpha$ and $\beta$. For this reason, we have decided to omit the proof.
\end{proof}

\section{A variety of other large scale asymptotic results}

In the final section of this paper, we switch to other important characteristics of the Generalized Gamma Polytope $K_\la$, and denote its $i$-th intrinsic volume and the number of $j$-dimensional faces by $V_i(K_\la)$, $i\in \{1,\ldots,d\}$, and $f_j(K_\la)$, $j\in \{0,\ldots,d-1\}$, respectively. In particular, $V_d(K_\la)$ represents the volume, while $f_0(K_\la)$ indicates the number of vertices.\\ 
Again, in the Gaussian case, i.e., $\alpha = 0$ and $\beta = 2$, these characteristics are well-studied objects in literature. One of the first issues taken into account concerned their expected values, as the number of points tends to infinity.
This line of research starts with the classical work of R\'enyi and Sulanke \cite{RenyiSulanke} in $1963$ and was continued by the paper of Affentranger \cite{Affentranger}, concerning, in particular, the face numbers and intrinsic volumes of Gaussian polytopes in higher dimensions. For all $i\in \{1,\ldots,d\}$ and $j \in \{0,\ldots, d-1\}$, it holds that 
\begin{align*}
\EE [V_i(K_\la)] \sim \binom{d}{i}\, \frac{\kappa_d}{\kappa_{d-i}}\, (2 \log \la)^{\frac{i}{2}} \qquad \text{and} \qquad \EE[f_j(K_\la)] \sim c_1\, \left(\log \la \right)^{\frac{d-1}{2}},
\end{align*}
as $\la \rightarrow \infty$, where $c_1\in (0,\infty)$ is an explicitly known constant only depending on $d$ and $j$, and $\kappa_d$ is the volume of the $d$-dimensional unit ball.
Here, for two functions $f(\la)$ and $g(\la)$, the notion $f(\la)\sim g(\la)$ indicates that, as $\la \rightarrow \infty$, $f(\la)/g(\la) \longrightarrow 1$.\\
%Hug and Reitzner \cite{HugReitzner} derived variance upper bounds and used them to establish laws of large numbers. For all $i \in \{1,\ldots, d\}$ and $j \in \{0, \ldots, d-1\}$, they proved that
%\begin{align*}
%\var[V_i(K_\la)] \le c_1\, (\log \la)^{\frac{i-3}{2}} \qquad \text{and} \qquad \var[f_j(K_\la)] \le c_2\, (\log \la)^{\frac{d-1}{2}},
%\end{align*}
%for sufficiently large $\la$, where $c_1,c_2\in (0,\infty)$ are constants only depending on $d$, $i$ and $j$. Matching lower bounds were obtained by B\'ar\'any and Vu \cite{BaranyVu} for all face numbers and the volume.\\
Hueter \cite{Hueter94,Hueter99} computed the precise variance asymptotics for the number of vertices and the volume of the Gaussian polytope $K_\la$, while Calka and Yukich \cite{CalkaYukich} generalized the result in their remarkable paper to hold for all intrinsic volumes and face numbers.
For all $i\in \{1,\ldots,d\}$ and $j\in \{0,\ldots,d-1\}$, they showed that
\begin{align*}
\var[V_i(K_\la)] \sim c_1\, (2 \log \la)^{i - \frac{d+3}{2}} \qquad \text{and} \qquad \var[f_j(K_\la)] \sim c_2\, (2 \log \la)^{\frac{d-1}{2}},
\end{align*}
as $\la\rightarrow \infty$, where $c_1\in [0,\infty)$ and $c_2\in (0,\infty)$ are constants only depending on $d$, $i$ and $j$. 
%In particular, the upper bound for the intrinsic volumes derived in \cite{HugReitzner} does not have the right order of magnitude. 
However, except for the case that $i=d$, Calka and Yukich were not able to exclude the possibility that $c_1=0$. Recently, B\'ar\'any and Th\"ale \cite{BaranyThäle} closed the missing gap and proved that, in fact, $c_1\in (0,\infty)$ for all other intrinsic volumes, too. We are able to extend the expectation and variance asymptotics to our class of Generalized Gamma Polytopes. %Our result reads as follows. 
\begin{thm}\label{EV}
	Let $i\in \{1,\ldots,d\}$ and $j\in \{0,\ldots,d-1\}$. Then, it holds that 
	\begin{align*}
	\EE[V_i(K_\la)] \sim \binom{d}{i} \frac{\kappa_d}{\kappa_{d-i}} (\be \log \la)^{\frac{i}{\be}} \qquad \text{and} \qquad \EE[f_j(K_\la)] \sim c_1\, (\be \log \la)^{\frac{d-1}{2}},
	\end{align*}
	as well as
	\begin{align*}
	\var[V_i(K_\la)] \sim c_2\, (\be \log \la)^{\frac{4i - \be(d+3)}{2\be}} \qquad \text{and} \qquad \var[f_j(K_\la)] \sim c_3\, (\be \log \la)^{\frac{d-1}{2}},
	\end{align*}
	as $\la \rightarrow \infty$, where $c_1,c_2,c_3\in (0,\infty)$ are constants only depending on $d$, $i$, $j$, $\alpha$ and $\beta$.
\end{thm}

\begin{remark}
Surprisingly, the constants $c_2$ and $c_3$ in the previous theorem are literally the same as the ones appearing in the Gaussian setup%(see \cite[Theorem 1.3 and Theorem 1.4]{CalkaYukich})
, and can be defined in terms of the limiting germ-grain processes $\Phi$ and $\Psi$ (see \cite[Theorem 2.1]{CalkaYukich}). Since it is known from \cite{BaranyThäle,BaranyVu} that these limiting constants are strictly positive, we do not have to prove positivity of $c_2$ and $c_3$ in our generalized setting. This is especially advantageous because proving positivity of variance asymptotics is a demanding task (see, for example, \cite{BaranyThäle,CalkaSchreiberYukich,Reitzner1,ReitznerCLT} for highly complicated computations of lower variance bounds in different random polytope models).
%	Let us draw attention to an interesting phenomenon, appearing in the case that $i=d$, $\be=4$ and arbitrary $\a > -1$. Then, the order of magnitude of the variance of the volume of $K_\la$ is completely independent of the dimension $d$. More formally, it holds that
%	\begin{align*}
%	\var[V_d(K_\la)] \sim c_1\, (4 \log \la)^{-\frac{3}{2}},
%	\end{align*} 
%	as $\la \rightarrow \infty$.
\end{remark}

The central limit problem for Gaussian polytopes has first been treated again by Hueter \cite{Hueter94,Hueter99} for the number of vertices and the volume, and been generalized in the breakthrough paper by B\'ar\'any and Vu \cite{BaranyVu} to hold for all other face numbers, too. Finally, B\'ar\'any and Th\"ale \cite{BaranyThäle} added the result for the lower-dimensional intrinsic volumes. Again, we are able to formulate a central limit theorem in our setting of Generalized Gamma Polytopes. Let $\mathcal{N}(0,1)$ denote a standard normal distributed random variable and $\stackrel{D}{\longrightarrow}$ convergence in distribution.
\begin{thm}\label{CLT}
For all $i\in \{1,\ldots,d\}$ and $j\in \{0,\ldots,d-1\}$, it holds that
\begin{align*}
\frac{V_i(K_\la) - \EE[V_i(K_\la)]}{\sqrt{\var[V_i(K_\la)]}}\ \stackrel{D}{\longrightarrow}\, \mathcal{N}(0,1) \qquad \text{and} \qquad \frac{f_j(K_\la) - \EE[f_j(K_\la)]}{\sqrt{\var[f_j(K_\la)]}}\ \stackrel{D}{\longrightarrow}\, \mathcal{N}(0,1),
\end{align*} 
as $\la \rightarrow \infty$.  
\end{thm}

Only recently, Grote and Th\"ale \cite{GroteThäle2} derived a number of other large scale asymptotic results for the intrinsic volumes and the face numbers in the Gaussian polytope setting, that we are able to generalize to arbitrary parameter $\alpha$ and $\beta$ in the underlying density of the Poisson point process $\cP_\la$. To keep the presentation short, we have decided to state the results just for the intrinsic volumes of $K_\la$, stressing that similar results hold for all face numbers, too. For a complete list of the results for the face numbers and more background material concerning the statements in the upcoming theorem, we refer to the dissertation of the author \cite[Section 3.4.1]{DissGrote}. 

\begin{thm}\label{CCC}
	Let $i\in \{1,\ldots,d\}$. Then, the following assertions are true.
	\begin{itemize}
%		\item[(a)](Central limit theorem with Berry-Esseen bound) We have that for sufficiently large $\la$,
%		\begin{align*}
%		\sup\limits_{y\in \mathbb{R}} \left|\PP \left(\frac{V_i(K_\la) - \EE[V_i(K_\la)]}{\sqrt{\var[V_i(K_\la)]}} \leq y\right) - \Phi(y) \right| \leq c_1\, (\log \la)^{-\frac{d-1}{4(4d + 2i + 9)}},
%		\end{align*}
%		where $c_1\in (0,\infty)$ is a constant only depending on $d$, $i$, $\alpha$ and $\beta$.
%		\item[(b)](Bound on the relative error in the central limit theorems) We have that for sufficiently large $\la$,
%		\begin{align*}
%		&\Bigg|\log{\PP\big(V_i(K_\la)-\EE [V_i(K_\la)]\geq y\, \sqrt{\var [V_i(K_\la)]}\,\big)\over 1-\Phi({y})}\Bigg| \leq c_1\,(1+y^3)\, (\log \la)^{-{d-1\over 4(4d + 2i +9)}},
%		\intertext{and}
%		&\Bigg|\log{\PP\big(\vol_i(K_\la)-\EE [\vol_i(K_\la)]\leq -y\, \sqrt{\var [\vol_i(K_\la)]}\,\big)\over \Phi(-{y})}\Bigg| \leq c_2\,(1+y^3)\, (\log \la)^{-{d-1\over 4(4d + 2i +9)}},
%		\end{align*}
%		for all 
%		\begin{align*}
%		0\leq y\leq c_3\, (\log \la)^{d-1\over 4(4d + 2i +9)}, 
%		\end{align*}
%		where $c_1,c_2,c_3\in (0,\infty)$ are constants only depending on $d$, $i$, $\alpha$ and $\beta$.
		\item[(a)](Concentration inequality) Let $y\geq 0$. Then, we have that for sufficiently large $\lambda$,
		\begin{align*}
		&\PP\big(|V_i(K_\lambda)-\EE [V_i(K_\la)]|\geq y\, \sqrt{\var [V_i(K_\la)]}\,\big)\\
		&\qquad \qquad \le 2\exp\Big(-{1\over 4}\min\Big\{{y^2\over 2^{2d + i +5}},c_1\,  (\log \la)^{d-1\over 4(2d + i+5)}\, y^{1\over 2d + i+5}\Big\}\Big),
		\end{align*}
		where $c_1\in (0,\infty)$ is a constant only depending on $d$, $i$, $\alpha$ and $\beta$.
		\item[(b)](Marcinkiewicz-Zygmund-type strong law of large numbers) Let $p > \frac{4i - \be(d+3)}{4i}$, and let $(\lambda_k)_{k\in\NN}$ be a sequence of real numbers defined by $\la_k:= a^k$, $a>1$. Then, as $k \rightarrow \infty$, it holds that
		\begin{align*}
		\frac{V_i(K_{\la_k}) - \EE [V_i(K_{\la_k})]}{(\log \la_k)^{p \frac{i}{\be}}} \longrightarrow 0,
		\end{align*}
		with probability one.
%		\item[(e)] (Moment bound) Let $k\in \NN$. Then, for sufficiently large $\la$, it holds that
%		$$
%		c_1\, c_2^k\, (\log \la)^{k\frac{i}{\be}} \le \EE[V_i(K_\la)^k] \le c_3\,c_4^k\, k!\, (\log \la)^{k\frac{i}{\be}},
%		$$
%		where 
%		%the upper bound holds for all $\be \ge \frac{2i}{d+1}$, and 
%		$c_1,c_2,c_3,c_4\in (0,\infty)$ are constants only depending on $d$, $i$, $\alpha$ and $\beta$.
		\item[(c)](Moderate deviation principle) Let $(a_\lambda)_{\lambda > 0}$ be a sequence of real numbers, satisfying
		\begin{align*}
		\lim\limits_{\lambda \rightarrow \infty} a_\lambda = \infty \qquad \text{and} \qquad \lim\limits_{\lambda \rightarrow \infty} a_\lambda\,  (\log \la)^{-\frac{d-1}{4(4d + 2i + 9)}} = 0.
		\end{align*}
		Then, the family $$\left(\frac{1}{a_\la} \frac{V_i(K_\la) - \EE[V_i(K_\la)]}{\sqrt{\var [V_i(K_\la)]}} \right)_{\la > 0}$$ satisfies a moderate deviation principle on $\RR$ with speed $a_\la^2$ and rate function $x^2/2$.
	\end{itemize}
\end{thm}

\begin{proof}[Proof of Theorem \ref{EV}, Theorem \ref{CLT} and Theorem \ref{CCC}]
	In contrast to the main result of this paper stated in Theorem \ref{Festoon}, we have decided to sketch the proofs of the results presented in this section, since they are almost the same as in the Gaussian setup, previously treated by Calka and Yukich \cite{CalkaYukich} and Grote and Th\"ale \cite{GroteThäle2}, respectively, slightly modified to our setting. In particular, the proof of the expectation asymptotic follows the one in \cite[Section 5.2]{CalkaYukich}, while the variance is handled as in \cite[Section 5.3]{CalkaYukich}. Moreover, the central limit theorem, as well as the results stated in the previous theorem, are achieved as in \cite[Section 4.2]{GroteThäle2}. For detailed proofs, we refer to the dissertation of the author (see \cite[Chapter 3]{DissGrote}).\\ %The main steps in the process of this generalization are the following.\\
	The starting point in the analysis is to write the intrinsic volumes and face numbers of $K_\la$ as a sum of so-called `score-functions' over all points from the Poisson process $\cP_\la$ (see \cite[Equation (3.21)]{DissGrote}), and to consider the measure valued version induced by the key geometric functionals, taking thereby care of their spatial profiles (see \cite[Equation (3.26)]{DissGrote}). For example, the $i$-th intrinsic volume of $K_\la$ can be expressed as
	\begin{align*}
	V_i(K_\la) = \sum_{x \in \cP_\la} \xi_{V_i}(x,\cP_\la)\, \delta_x,
	\end{align*}
	where $\delta_x$ is the Dirac-measure at $x$, and $\xi_{V_i}$ abbreviates some score-function depending on the interplay of $x$ with the complete point set $\cP_\la$ (see \cite[Page 85]{DissGrote}). Then, these score-functions are analyzed further. In particular, one needs to derive localization results (see \cite[Section 3.2.1]{DissGrote}) and moment estimates (see \cite[Section 3.3.2]{DissGrote}).\\
	By using these very technical preparations, the proofs of the expectation and variance asymptotics in the setting of Generalized Gamma Polytopes are worked out in detail in \cite[Section 3.5.1 and Section 3.5.3]{DissGrote}. Once more based on the above mentioned preparations, the proofs of the statements in Theorem \ref{CLT} and Theorem \ref{CCC} rely on a precise cumulant estimate for the intrinsic volumes and face numbers of $K_\la$ (see \cite[Section 3.3 and Section 3.5.2]{DissGrote}). In a next step, the central limit theorem and the results in part $(a)$ and $(c)$ of Theorem \ref{CCC} are direct consequences of this cumulant estimate in combination with results from \cite{DoeringEichelsbacher,EichelsbacherSchreiberRaic,SaulisBuch}, summarized for example in \cite[Lemma 5.10]{GroteThäle} or \cite[Lemma 4.2]{GroteThäle2}, while for the proof of part $(b)$ of Theorem \ref{CCC}, we cite \cite[Page 166]{DissGrote}.%, while for $(e)$ we refer to \cite[Page 168]{DissGrote}.
\end{proof}

\end{document}